\documentclass[a4paper,12pt]{article}
\usepackage[most]{tcolorbox}
\newtcbox{\mybox}[1][]{enhanced, colframe=black, colback=orange!45, 
       nobeforeafter, tcbox raise base, shrink tight, extrude by=1mm, #1}
\usepackage{tikz} 
\usetikzlibrary{arrows.meta, positioning, calc} %
\usepackage{inputenc}
\usepackage{graphics}
\usepackage{epsfig}
\usepackage{graphicx}
\usepackage{multirow}
\usepackage{latexsym}
\usepackage{graphics}
\usepackage{multirow}
\usepackage{tikz}
\usepackage{epsfig}
\usepackage{amsmath,amssymb,amsfonts,theorem,color,bm}
\usepackage{multirow}
\usepackage{verbdef}
\usepackage{mathrsfs}
\usepackage{hyperref}
\usepackage{subcaption}
\usepackage{epstopdf}
\usepackage{float}
\newcommand{\bs}{\boldsymbol}

\newcommand{\beqn}{\begin{equation}}
\newcommand{\eeqn}{\end{equation}}

\newtheorem{remark}{Remark}[section]

\newtheorem{corollary}{Corollary}[section]
\newtheorem{theorem}{Theorem}[section]

\newtheorem{proposition}[theorem]{Proposition}

\def\blackbox{\leavevmode\vrule height 5pt width 4pt depth 0pt\relax}
\newenvironment{proof}{\begin{trivlist}
\item[]\hspace{0cm}{\bf Proof:}
\hspace{0cm} }{\hfill $\blackbox$
\end{trivlist}}

\def\R{{\mathbb R}}

\definecolor{gray1}{gray}{0.25}

\newcommand\restr[2]{{
\left.\kern-\nulldelimiterspace 
#1 
\right|_{#2} 
}}

\newcommand{\mK}{\mathsf{K}}

\newcommand{\mW}{\mathsf{W}}

\def\R{{\mathbb R}}

\newcommand{\norm}[1]{\Vert{#1}\Vert}	

\renewcommand{\v}{\ensuremath{\boldsymbol}}	

\begin{document}
\title{A Recipe for Learning Variably Scaled Kernels via Discontinuous Neural Networks}

\author{
	G. Audone$^{\dagger,}$\thanks{Dipartimento di Scienze Matematiche \lq\lq Giuseppe Luigi Lagrange\rq\rq, Politecnico di Torino, Italy \\
    $^{\dagger}$Gruppo Nazionale per il Calcolo Scientifico INdAM, Italy}
	\quad F. Della Santa$^{\dagger,*}$
	\quad E. Perracchione$^{\dagger,*}$
    \quad S. Pieraccini$^{\dagger,*}$\\ 
{\tt gianluca.audone@polito.it; francesco.dellasanta@polito.it}\\ {\tt emma.perracchione@polito.it; sandra.pieraccini@polito.it}
}		

\maketitle

\section*{Abstract}
The efficacy of interpolating via Variably Scaled Kernels (VSKs) is known to be dependent on the definition of a \emph{proper} scaling function, but no numerical recipes to construct it are available. Previous works suggest that such a function should mimic the target one, but no theoretical evidence is provided. This paper fills both the gaps: it proves that a scaling function reflecting the target one may lead to enhanced approximation accuracy, and it provides a user-independent tool for learning the scaling function by means of Discontinuous Neural Networks ($\delta$NN), i.e., NNs able to deal with possible discontinuities. Numerical evidence supports our claims, as it shows that the key features of the target function can be clearly recovered in the learned scaling function.

\section{Introduction}

    Variably Scaled Kernels (VSKs) were introduced in 2015 \cite{Bozzini2015199} with the main purpose of improving the stability of kernel interpolants thanks to the definition of a scaling function which was originally thought of as a continuous version of the shape parameter, whose value affects accuracy and stability indicators (see e.g. \cite{driscoll2002interpolation,fornberg2004stable,ling2022stochastic,marchetti2021extension,rippa1999algorithm}). Later works show that the use of VSKs might be valuable when the target function of the scattered data interpolation problem is characterized by discontinuities. Indeed in such cases, defining a scaling function with the same key features of the target one, e.g. the same discontinuities, results in a VSK basis sharing the same discontinuities with the target; e.g., we refer the reader to \cite{vskmpi,romani,rossini_drna}. This fact suggests using VSKs with scaling functions which somehow mimic the target ones as, intuitively, this allows to directly insert in the kernel basis itself some prior information about the sought solution, and especially in applied fields, this turns out to be a relevant issue \cite{Perracchione2023}. Nevertheless, this intuition lacks both theoretical and numerical evidence.

    This work addresses these issues by providing theoretical justification for the claim, demonstrating by means of  the Lebesgue functions that a scaling function reflecting the behaviour of the target may lead to enhanced approximation accuracy. As a further confirmation of our claim, we propose a user-independent way of choosing the scaling function by training a Discontinuous Neural Network ($\delta$NN) able to learn also possibly discontinuous scaling functions directly from data \cite{DellaSanta2023}. Our results confirm the theoretical findings: the learned scaling function closely resembles the target function. This data-driven approach, namely in what follows $\delta$NN-VSKs, offers a user-independent and adaptive solution for meshfree approximation tasks. 

    The paper is organised as follows. In Section \ref{notation} we recall the basics of kernel-based approximation. Section \ref{vskdelta} is the core of the paper presenting the theory and the numeric intuition behind the $\delta$NN-VSKs. Experiments and conclusions are respectively offered in Sections \ref{esperimenti} and \ref{Conclusion}.

\section{Brief review of kernel-based interpolation}
\label{notation}

    We consider symmetric reproducing kernels $\kappa: \Omega \times  \Omega \longrightarrow \mathbb{R}$ for $$H_{\kappa} = \mathrm{span} \left\{ \kappa(\cdot,\boldsymbol{x}), \hskip 0.2cm \boldsymbol{x} \in \Omega\subset\mathbb{R}^d \right\},$$ which is a pre-Hilbert space equipped with an inner product $\left(\cdot,\cdot\right)_{H_{\kappa} }$. Those are kernels so that (refer e.g. to \cite[Definition 2.6, p. 32]{Fasshauer15})  $\kappa(\boldsymbol{x}, \boldsymbol{z}) = \kappa(\boldsymbol{z}, \boldsymbol{x})$, $\boldsymbol{x},\boldsymbol{z} \in \Omega$, with $\kappa(\cdot, \boldsymbol{x}) \in H_{\kappa}$, and 
    $$
    \left(f,\kappa(\cdot, \boldsymbol{x})\right)_{H_{\kappa} } = f(\boldsymbol{x}), \quad \boldsymbol{x}\in \Omega.
    $$
    The native space ${\cal	N}_{\kappa}$ of the kernel $\kappa$ is then defined as the completion of $H_\kappa$ with respect to the norm $||\cdot||_{H_{\kappa}} = \sqrt{(\cdot, \cdot)_{H_{\kappa}}}$, and hence for all $f \in H_\kappa$ we have that $||f||_{{\cal	N}_{\kappa}}= ||f||_{H_{\kappa}}$; here and throughout all this section, we refer the reader to the monographs \cite{Fasshauer15,Wendland05} for more details. 

    A common subclass of reproducing kernels, which is frequently used in the framework of scattered data interpolation problems, is given by radial kernels. To any radial kernel, we are able to uniquely associate a Radial Basis Function (RBF) $ \phi: \mathbb{R}_{+} \longrightarrow\mathbb{R}$, where $\mathbb{R}_{+}= [0,+\infty)$, and (possibly) a shape parameter $\varepsilon>0$ such that, for all $\boldsymbol{x},\boldsymbol{z}
    \in \Omega$,
    \begin{equation*} 
        \kappa(\boldsymbol{x},\boldsymbol{z})= \kappa_{\varepsilon}(\boldsymbol{x},\boldsymbol{z})=  \phi( \varepsilon ||\boldsymbol{x}-\boldsymbol{z}||_2)=\phi_{\varepsilon}( ||\boldsymbol{x}-\boldsymbol{z}||_2)=  \phi(r),
    \end{equation*}
    where $r=||\boldsymbol{x}-\boldsymbol{z}||_2$. For the sake of notation simplicity, we may omit the dependence of the kernel on the shape parameter, which is also referred to as scale parameter in the machine learning literature.

    In order to formally introduce the scattered data interpolation problem, we consider a function $f:\Omega\longrightarrow\mathbb{R}$, with $\Omega\subset\mathbb{R}^d$, and an associated set of {functional} values ${F}_n=\{ f(\boldsymbol{x_i})\}_{i=1}^{n}$ sampled at a data nodes set  ${{X}_n=\{\boldsymbol{x}_i\}_{i=1}^{n}\subset \Omega}$.   Given these sets, the goal consists in computing an approximation of the unknown function $f$, called $P_f$, by imposing the $n$ interpolation constraints, i.e., 
    \begin{equation}\label{interp_cond}
    P_{f}(\boldsymbol{x}_i) = f(\boldsymbol{x_i}), \quad i= 1,\ldots, n.
    \end{equation}
    Assuming that the interpolating function $P_f$ can be written as
    \begin{equation}
        \label{eq1}
        P_f\left( \boldsymbol{x}\right)= \sum_{k=1}^{n} c_k \kappa \left( \boldsymbol{x} , \boldsymbol{x}_k  \right), \quad \boldsymbol{x} \in \Omega,
    \end{equation}
    the coefficients $\{c_i\}_{i=1}^{n}$ are the solution of the linear system 
    $$\mK \boldsymbol{c} = \boldsymbol{f},$$
    where $  \boldsymbol{c}= \left(c_1, \ldots,
    c_n\right)^{\intercal}$, $  \boldsymbol{f} =\left(f_1, \ldots , f_n\right)^{\intercal}$, being $f_i = f(\boldsymbol{x}_i)$, and $
    \mK_{ik}= \kappa \left( \boldsymbol{x}_i , \boldsymbol{x}_k  \right),$ $i,k=1, \ldots, n$.
    Letting $\kappa$ be a symmetric and strictly positive definite kernel, the collocation matrix is positive definite, and equivalently to \eqref{eq1}, we may write the nodal form of the interpolant as  
    \begin{equation*}
        P_f\left( \boldsymbol{x}\right)=  \boldsymbol{\kappa}^{\intercal}({\boldsymbol{x}} ) \mK^{-1} \boldsymbol{f}, \quad \boldsymbol{x} \in \Omega,
    \end{equation*}
    where $$
    \boldsymbol{\kappa}^{\intercal}({\boldsymbol{x}} ) = \left(\kappa(\boldsymbol{x},\boldsymbol{x}_1), \ldots,\kappa (\boldsymbol{x},\boldsymbol{x}_n) \right).$$

    {Even if here we focus on interpolation, we point out that the kernel theory is not limited to interpolation, indeed the Ansatz conditions could be generalized to any linearly independent set of continuous linear functionals $\{ \chi_1, \ldots, \chi_n \}$, so that equation \eqref{interp_cond} becomes $\chi_i P_{f} = \chi_i f,$ $i= 1,\ldots, n.$}
    
\section{Investigating and learning the scaling function for VSKs}
\label{vskdelta}

    The idea behind VSKs, first introduced in the interpolation context in \cite{Bozzini2015199} and later adapted to least squares approximation \cite{Esfahani202338},  is to introduce a new kernel basis. {This basis can be interpreted as a kind of feature augmentation strategy. Specifically, the original aim was to replace the shape parameter of the basis function with a continuous shape function. The authors of \cite{Bozzini2015199} demonstrated that this approach is equivalent to defining a standard kernel in a higher-dimensional space—effectively adding features via a scaling function denoted here by $\bar{f}$—and then projecting it back onto the original dimension.}

    {More formally, let} $\Sigma \subset  \mathbb{R}^m$, $m>0, m \in \mathbb{N}$, and let $\kappa: \tilde{\Omega}\times \tilde{\Omega}\longrightarrow\R$,  {$\tilde{\Omega}=\Omega \times \Sigma \subset  \mathbb{R}^{d+m}$}, be a continuous radial positive definite kernel. Given a scaling function $\bar{f}: \Omega \longrightarrow \Sigma,$ a VSK $\kappa_{\bar{f}}: \Omega  \times \Omega \longrightarrow \mathbb{R}$ is defined as
    \begin{equation*}
        \kappa_{\bar{f}}(\bs{{x}},\bs{{z}})=\kappa\left((\bs{{x}},\bar{f}(\bs{x})),(\bs{z},\bar{f}(\bs{z}))\right)=:\kappa(\tilde{\bs{{x}}},\tilde{\bs{{z}}}),
    \end{equation*}
    for $\bs{x},\bs{z}\in\Omega$. Then, the VSK interpolant is simply given by $P^{{\bar{f}}}_f\left({\bs{x}}\right) := P_f\left(\tilde{\bs{{x}}}\right)$,
    or, equivalently, 
    \begin{equation}\label{eq:VSK_interpol}
         P^{{\bar{f}}}_f\left({\bs{x}}\right) :=   \sum_{k=1}^{n} c_k \kappa_{{\bar{f}}}( {\bs{x}},{\bs{x}}_k)  = \sum_{k=1}^{n} c_k \kappa(\tilde{\bs{{x}}},\tilde{\bs{{x}}}_k).
    \end{equation}

    {As radial kernels only depend on the distance of the nodes}, the VSK interpolant turns out to be easy to implement. Indeed, the coefficients $c_1, \ldots, c_n$ of the interpolant $P_f^{{\bar{f}}}$ are computed by solving
    \begin{equation}\label{sistema}
          \underbrace{\begin{pmatrix}
	       \kappa(\tilde{\bs{x}}_1,\tilde{\bs{x}}_1) & \cdots & \kappa(\tilde{\bs{x}}_1,\tilde{\bs{x}}_n)\\
	       \vdots & & \vdots\\
	       \kappa(\tilde{\bs{x}}_n,\tilde{\bs{x}}_1) & \cdots & \kappa( \tilde{\bs{x}}_n,\tilde{\bs{x}}_n)\\ 
	\end{pmatrix}}_{\mK_{{\bar{f}}}}
    \begin{pmatrix}
        c_1 \\ \vdots \\ c_n 
    \end{pmatrix}
        =
    \begin{pmatrix}
        f_1  \\ \vdots \\ f_n
    \end{pmatrix}.
    \end{equation} 
    {In the discrete setting, the approximant is then evaluated at $N$ points denoted by $\Xi = \{\bs{\xi}_k\}_{k=1}^{N}$. Practically, we consider $m=1$, and hence the VSK interpolant, after defining the augmented sets of nodes 
    \begin{equation}\label{augmented}
        \{\tilde{\bs{x}}_i\}_{i=1}^n = \{(\bs{x}_i, \bar{f}(\bs{x}_i))\}_{i=1}^n, \quad \textrm{and} \quad \{\tilde{\bs{\xi}_k}\}_{k=1}^{N} = \{(\bs{\xi}_k, \bar{f}(\bs{\xi}_k))\}_{k=1}^{N} \subset \mathbb{R}^{d+1}, 
    \end{equation}
    and computing the coefficients as in \eqref{sistema}, is defined as the projection on $\mathbb{R}^d$ of a standard interpolant in $\mathbb{R}^{d+1}$, i.e,  
    \begin{equation}
    \label{interpolant}
        P_{f}^{\bar{f}}(\bs{\xi}_k) = P_{f} (\tilde{\bs{\xi}}_k) = \sum_{i=1}^n c_i \kappa(\tilde{\bs{\xi}}_k, \tilde{\bs{x}}_i),
    \end{equation}
    with $k=1,\ldots,N$.}

    {The advantages of using VSKs, i.e., their easy implementation and the fact that we can superimpose some prior information, are briefly summarized in the following remark.}

    \begin{remark}\label{remark0}
        Because of their easy implementation, VSKs have already been used in many applications whose goals consist in recovering functions with steep gradients,  high oscillations and/or discontinuities. In case of discontinuities, intuitively, if the scaling function ${\bar f}$ is chosen so that it has discontinuities at the same edges of $f$, so is for the radial kernel $\kappa_{\bar f}$ and hence for $P_f^{\bar{f}}$. We refer the reader to \cite{vskmpi} for further details on the characterization of native spaces induced by the discontinuous kernels. 
    \end{remark}

    Hence, following the intuition of Remark \ref{remark0}, we are interested in proposing bounds that highlight the dependence of the error on the scaling function (see Subsection \ref{sec:theoretic_res}) and in providing an algorithm to automatically define such a scaling function (see Subsection \ref{sec:delta}). {In order to make the presentation clearer we want to stress the fact that the scheme used to find the nearly optimal scaling function $\bar{f}$, based on the so-called $\delta$NNs, could be thought of as an independent step of the analysis carried out in the next subsection. Indeed, the use of $\delta$NNs is primarily due to numerically support the claim that the scaling function should resemble the target one, even if we do not have theoretical evidence about the fact that such choice is definitely the best one. Nevertheless, independently of the results reported in the next subsection, our $\delta$NN algorithm returns a learned function $\bar{f}$ that definitely has the same key features of the target.}

    {Before going into details and in order to set our final goal, we summarize, in Flowchart 1, the main steps needed to compute the VSK interpolant (well-known in literature already), with the difference that here the scaling function will be no longer an input of the scheme but the result of the $\delta$NN optimization procedure.}

    \begin{figure}[ht]
    \centering
    \begin{tikzpicture}[
      node distance=2.8cm and 2.5cm,
      >=Latex,
      every node/.style={align=center, font=\small},
      arrow/.style={draw, ->, thick},
      inputnode/.style={draw, fill=blue!20, circle, minimum size=2cm},
      processnode/.style={draw, fill=green!20, rectangle, minimum height=1.2cm, minimum width=2.5cm},
      optionalnode/.style={draw, fill=orange!20, rectangle, minimum height=1.2cm, minimum width=2.5cm},
      outputnode/.style={draw, fill=red!20,  rectangle, minimum height=1.2cm, minimum width=2.5cm}
        ]

    \node[inputnode] (input) at (0, 0) {\textbf{Inputs}:\\
        Nodes \\Functional values \\ Evaluation points \\ Kernel\\
    };

    \node[processnode, right=of input] (augmented) {Define the augmented sets \\of points as in \eqref{augmented}};

    \node[processnode, below=of augmented] (matrix) {Define the VSK matrix, and \\
    compute the coefficients as in \eqref{sistema}};

    \node[outputnode] (interpolant) at (0, -4) {\textbf{Outputs}:\\
    Evaluation of the VSK \\interpolant as in \eqref{interpolant}};

    \draw[arrow, bend left=60] (input) to node[above]  {\mybox[extrude bottom by=0.1cm]{Define the scaling function}
    }
    (augmented);

    \draw[arrow] (augmented) -- (matrix);
    \draw[arrow] (matrix) -- (interpolant);

    \end{tikzpicture}
    \captionsetup{labelformat=empty, name=Diagramma} 
    \caption{\textbf{Flowchart 1:} {Diagram for the VSK-based algorithm.}}
    \end{figure}
    \setcounter{figure}{0}

\subsection{VSKs: investigating the scaling function via Lagrange bases}\label{sec:theoretic_res}

    Following the insights of Remark \ref{remark0}, the authors of previous papers (e.g., see \cite{vskmpi,rossini_drna,Perracchione2023}) claim that if the scaling function ${\bar{f}}$ \emph{somehow mimics} the target ${{f}}$ or some of its key features, then the VSK interpolant outperforms the standard one. Nevertheless, no theoretical evidence has been provided and no numerical tests were conducted to determine the \emph{optimal} scaling function. For the first item, we point out that for any set of distinct points $\{\boldsymbol{x}_i\}_{i=1}^n$, there exist functions (known as Lagrange or cardinal bases) $\varphi^{\bar{f}}_j \in \textrm{span} \{\kappa_{\bar{f}}(\cdot,{\boldsymbol{x}}_j), j=1,\ldots,n\}$ so that the VSK interpolant can be written as
    \begin{align*}
           P_{f}^{\bar{f}}(\boldsymbol{x}) = \sum_{j=1}^n f_j \varphi_{j}^{\bar{f}}(\boldsymbol{x}),
    \end{align*}
    being $\varphi_{j}^{\bar{f}}(\boldsymbol{x}):=\varphi_{j}(\tilde{\boldsymbol{x}})$. Such functions are computed by solving the linear system 
$$
\mK_f \tilde{\boldsymbol{\varphi}}=\tilde{\boldsymbol{\kappa}},
$$ 
where $\tilde{\boldsymbol{\varphi}} = (\varphi_1(\tilde{\boldsymbol{x}}),\ldots,\varphi_n(\tilde{\boldsymbol{x}}))^{\intercal}$ and the vector on the right-hand side is given by $\tilde{\boldsymbol{\kappa}} = (\kappa(\tilde{\boldsymbol{x}},\tilde{\boldsymbol{x}}_1),\ldots,\kappa(\tilde{\boldsymbol{x}},\tilde{\boldsymbol{x}}_n))^{\intercal}$.

Such Lagrange basis allows us to formally introduce error bounds for the VSK interpolant that depend on the \emph{similarities} between $\bar{f}$ and $f$.  

    \begin{proposition}     \label{prop_lagr} 
            {Let ${\bar{f}}$ be a scaling function and let $\kappa_{\bar{f}}: \Omega \times \Omega \longrightarrow \R $ be the VSK corresponding to the strictly positive definite and symmetric kernel $\kappa: \tilde{\Omega} \times \tilde{\Omega} \longrightarrow \R$. Let $P_{f}^{\bar{f}}: \Omega \longrightarrow \mathbb{R}$ be the VSK interpolant}. Then, the following point-wise error bound holds true
    \begin{align} \label{eq:lagrange}
          \left|\left(f - P_{f}^{\bar{f}}\right)(\boldsymbol{x})\right| & \leq \left|\left(f - P_{\bar{f}}^{\bar{f}}\right)(\boldsymbol{x})\right| + \| \bs{f} - \bs{\bar{f}}  \|_{\infty} \lambda_{\bar{f}}(\boldsymbol{x}), 
          \end{align}
          where $$\lambda_{\bar{f}}(\boldsymbol{x}) = \sum_{j=1}^{n} | \varphi^{\bar{f}}_j (\boldsymbol{x}) |, $$
    is the Lebesgue function for the VSK interpolant, {and $\boldsymbol{f}-\bs{\bar{f}} = (f(\boldsymbol{x}_1)-\bar{f}(\boldsymbol{x}_1), \ldots, f(\boldsymbol{x}_n)-\bar{f}(\boldsymbol{x}_n))^{\intercal}=(f_1-\bar{f}_1, \ldots, f_n-\bar{f}_n)^{\intercal}$}.
    \end{proposition}
    \begin{proof}
          We trivially have that 
        \begin{equation*}
        \begin{split}
    \left|\left(f - P_{f}^{\bar{f}}\right)(\boldsymbol{x})\right| & = \left|\left(f - P_{\bar{f}}^{\bar{f}}\right)(\boldsymbol{x})+\left(P_{\bar{f}}^{\bar{f}} - P_{f}^{\bar{f}}\right)(\boldsymbol{x})\right| \\
     & \leq \left|\left(f - P_{\bar{f}}^{\bar{f}}\right)(\boldsymbol{x})\right|+\left|\left(P_{{f}}^{\bar{f}} - P_{\bar{f}}^{\bar{f}}\right)(\boldsymbol{x})\right|. 
    \end{split}
    \end{equation*}
    To bound the second term in the above inequality, we observe that 
    \begin{align*}
    \left|\left(P_{{f}}^{\bar{f}} - P_{\bar{f}}^{\bar{f}}\right)(\boldsymbol{x})\right| & = \left|P^{\bar{f}}_{f-{\bar{f}}}(\boldsymbol{x})\right| = \sum_{j=1}^n | (f_j-{\bar{f}}_j)  \varphi^{\bar{f}}_j(\boldsymbol{x}) | \leq  \| \bs{f}-\bs{\bar{f}} \|_{\infty} \lambda_{\bar{f}}(\boldsymbol{x}), 
    \end{align*}
    and the thesis follows.
\end{proof}

Equivalently to \eqref{eq:lagrange}, we might write 
    \begin{align*}
       \left|\left(f - P_{f}^{\bar{f}}\right)(\boldsymbol{x})\right| & \leq \left|\left(f - P_{\bar{f}}^{\bar{f}}\right)(\boldsymbol{x})\right| + \| \bs{f} - \bs{\bar{f}}  \|_{\infty} \Lambda_{\bar{f}}, 
    \end{align*}
    where $\Lambda_{\bar{f}}$ is the Lebesgue constant given by \cite{Brutman}
    \[
        \Lambda_{\bar{f}} = \sup_{\boldsymbol{x} \in \mathbb{R}^d} \sum_{j=1}^{n} | \varphi_j^{\bar{f}} (\boldsymbol{x}) |.
    \]    
    {We point out that to get rid of the second term in the right-hand side of \eqref{eq:lagrange}, we only need to set the nodal values of $\bar{f}$ equal to the ones of $f$;  this could be trivially achieved since the interpolation conditions are known.} Nevertheless, the first term suggests that the scaling function $\bar{f}$ should be close to $f$ on the whole domain $\Omega$, as in this way $P_{\bar{f}}^{\bar{f}}$ should approach $P_{f}^{\bar{f}}$ and $f$ in turn, following classical error bounds based either on the power function or on the fill-distance (see \cite[Theorem 11.4 and Theorem 11.13]{Wendland05}). {Moreover, we stress that taking $f\equiv \bar{f}$ is an unrealistic choice, as $f$ is unknown, and we do not have any theoretical nor numerical evidence that $|(f-P_f^f)(\boldsymbol{x}) | \leq |(f-P_f^{\bar{f}})(\boldsymbol{x})| $ for each $\bar{f}$ and $\boldsymbol{x} \in \Omega$.} 
    
    Finally, in terms of native spaces error bounds, we have the following corollary. 

    \begin{corollary}
        Let $\kappa_{\bar{f}}: \Omega \times \Omega \longrightarrow \R $ be the VSK associated to $\kappa: \tilde{\Omega} \times \tilde{\Omega} \longrightarrow \R$ and $f \in N_{\kappa}^{\bar{f}}$, then
        \begin{align*} 
                \left|\left(f - P_{f}^{\bar{f}}\right)(\boldsymbol{x})\right| & \leq \left(\|f\|_{N_{\kappa}^{\bar{f}}} + \boldsymbol{\bar{f}}^{\intercal} \mK_{{\bar{f}}}^{-1} \boldsymbol{\bar{f}}  \right)\| \kappa_{\bar{f}}\|_{N_{\kappa}^{\bar{f}}} + \| \bs{f}-\bs{\bar{f}}  \|_{\infty} \lambda_{\bar{f}}(\boldsymbol{x}). 
        \end{align*}
\end{corollary}
\begin{proof}
    In \cite[Theorem 2]{Bozzini2015199} the authors proved that the native spaces $N_{\kappa}$ and $N_{\kappa}^{\bar{f}}$ are isometric. Then, by the Cauchy-Schwarz inequality and the reproducing property, we have that  
        \begin{equation*}
        \begin{split}
    \left|\left(f - P_{\bar{f}}^{\bar{f}}\right)(\boldsymbol{x})\right| & = \left( f - P_{\bar{f}}^{\bar{f}} \ , \ \kappa_{\bar{f}} (\cdot,\boldsymbol{x}) \right)_{N_{\kappa}^{\bar{f}}} \\
    & \leq \| f - P_{\bar{f}}^{\bar{f}}\|_{N_{\kappa}^{\bar{f}}}  \| \kappa_{\bar{f}}\|_{N_{\kappa}^{\bar{f}}} \\
     & \leq \left( \| f \|_{N_{\kappa}^{\bar{f}}} + \|P_{\bar{f}}^{\bar{f}}\|_{N_{\kappa}^{\bar{f}}} \right) \| \kappa_{\bar{f}}\|_{N_{\kappa}^{\bar{f}}}, 
    \end{split}
    \end{equation*} 
    and as $\|P_{\bar{f}}^{\bar{f}}\|_{N_{\kappa}^{\bar{f}}}=\boldsymbol{\bar{f}}^{\intercal} \mK_{{\bar{f}}}^{-1} \boldsymbol{\bar{f}}$, the thesis follows.
\end{proof}

    We conclude this subsection by pointing out that we have some theoretical evidence of the fact that the function $\bar{f}$ should approximate ${f}$. However, it is worth clarifying that no optimal properties have been proved and, as the function $f$ is in practice unknown, we do not have any criteria or algorithm to automatically define the \emph{best} scaling function. Hence we implement a NN that confirms our findings and returns a safe (i.e., nearly-optimal) scaling function $\bar{f}$. More precisely, numerically we will observe that for the practical implementation of the $\delta$NN-VSKs the demanding pointwise convergence of the function $\bar{f}$ to ${f}$ can typically be relaxed, as reproducing some of the key features of the target usually leads to accurate approximations. Further comments are provided later in Subsecion \ref{sintetici}.

\subsection{$\delta$NN-VSKs: learning nearly-optimal scaling functions via $\delta$NNs}
\label{sec:delta}

    The NN that we consider is referred to as $\delta$NN and its first usage \cite{DellaSanta2023} was devoted to detect the discontinuity interfaces while approximating discontinuous functions. Such NN is characterized by the introduction of learnable discontinuities in the layers, legitimated by the universal approximation theorems \cite{Kidger20202306}, and hence might return a discontinuous output. {Before introducing the discontinuous layers that characterize $\delta$NNs, we recall that a standard Fully-Connected (FC) layer, from $n_{\rm in}$ units to $n_{\rm out}$ units, can be described by the function 
\begin{equation}\label{eq:FCL}
	\mathcal{L}(\v{x}) = \v{\sigma} \left( \mW^{\intercal} \v{x} + \v{b} \right)\,, \quad   \v{x}\in\R^{n_{\rm in}}\,,
\end{equation}
    where $\mW\in\R^{n_{\rm in}\times n_{\rm out}}$ is the weight matrix, $\v{b}\in\R^{n_{\rm out}}$ is the bias vector, and $\v{\sigma}$ is the element-wise application of the activation function $\sigma:\R\rightarrow\R$.
    
    The characterizing function of a discontinuous layer is a function $\delta\mathcal{L}:\R^{n_{\rm in}}\rightarrow\R^{n_{\rm out}}$ such that
\begin{equation}\label{eq:deltaL}
    \mathcal{L}(\v{x}) = \v{\sigma} \left( \mW^{\intercal} \v{x} + \v{b} \right) + \v{\alpha} \odot \v{\mathcal{H}} \left( \mW^{\intercal} \v{x} + \v{b} \right)\,, \quad   \v{x}\in\R^{n_{\rm in}}\,,
\end{equation}
    where $\v{\mathcal{H}}$ is the element-wise application of the Heaviside function and $\v{\alpha}\in\R^{n_{\rm out}}$ is the vector of {\it trainable discontinuity jumps}. In brief, for each $j=1,\ldots, n_{\rm out}$, equation \eqref{eq:deltaL} is equivalent to add to the $j$-th element of \eqref{eq:FCL} a jump of height $\alpha_j$ if the $j$-th row of $\mW^{\intercal}\v{x}+\v{b}$ is non-negative (otherwise, no jump is applied). See Figure \ref{fig:examplejumps} for a visual example of discontinuous layers with $n_{\rm in}=n_{\rm out}=1$.
}

\begin{figure}[htb]
	\centering
	\includegraphics[width=0.45\textwidth]{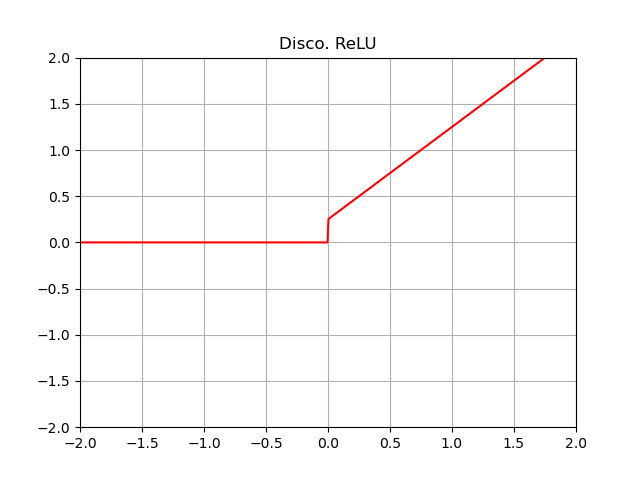}
	\includegraphics[width=0.45\textwidth]{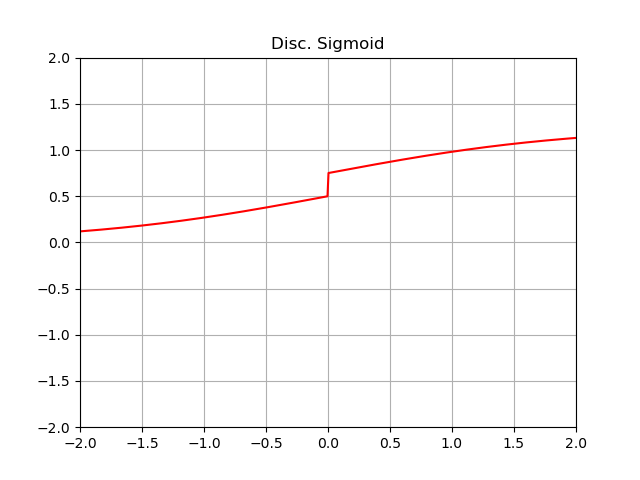}
	\\
	\quad
	\\
	\includegraphics[width=0.45\textwidth]{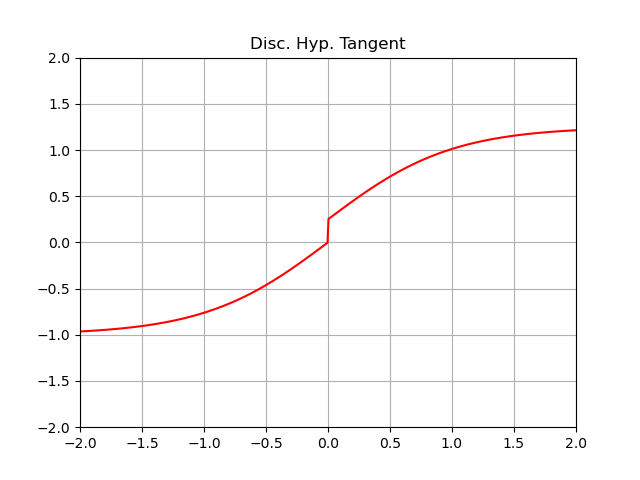}
	\includegraphics[width=0.45\textwidth]{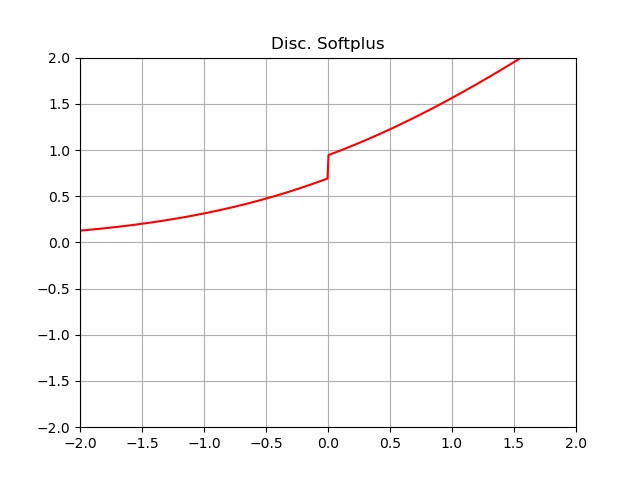}
	\caption{{Examples of discontinuous layers ($n_{\rm in}=n_{\rm out}=1$) with different activation functions.}}
	\label{fig:examplejumps}
\end{figure}

{The main property of a NN embedded with discontinuous layers (i.e., a $\delta$NN) is the ability to return both an approximation of the target function and an approximation of its discontinuity interfaces. Specifically, in a $\delta$NN, the weight matrices and the bias vectors of the discontinuous layers have the role of learning the discontinuity interfaces, while the discontinuity jump parameters have the role of learning the height of the jump along the discontinuity interfaces. For further theoretical or practical details about $\delta$NNs, we refer the reader to \cite{DellaSanta2023}.
}

{The reason why we are interested in $\delta$NNs for defining the scaling function $\bar{f}$ lies in the fact that, in principle, $\bar{f}$ can be discontinuous}, and we already have numerical evidence of the fact that if $\bar{f}$ mimics the discontinuities of $f$ then the Gibbs phenomenon observed in the approximations is drastically reduced. The novelty in this case is that we do not learn the discontinuities of $f$, we learn directly the function $\bar{f}$ without imposing any constraints on its discontinuities. 

We define the $\delta$NN as a parametric scaling function for the VSKs, aiming to find the \emph{optimal} one for the interpolant of $f$; i.e., we train a $\delta$NN $\bar{f}_{\v{\theta}}:\R^d\rightarrow\R^m$, with trainable parameters denoted by $\v{\theta}\in\R^p$, and we adapt the interpolation coefficients $\v{c}\in\R^n$ by minimizing the interpolation error of the VSK interpolant. 

More precisely, we build a NN model $\mathcal{M}_{\v{\theta},\v{\eta}}:\R^d\rightarrow\R^m\times\R^n$ pairing the $\delta$NN $\bar{f}_{\v{\theta}}$ and a model $\v{c}_\eta:=\v{c}(\v{\eta})$ for the coefficients, with $\v{\theta}\in\R^p, \v{\eta}\in\R^q$ trainable parameters, such that
\begin{equation*}
    \mathcal{M}_{\v{\theta},\v{\eta}}(\v{x})=(\bar{f}_{\v{\theta}}(\v{x}), \v{c}_{\eta}) ,
\end{equation*}
for each $\v{x}\in\R^d$.

The training of $\mathcal{M}_{\v{\theta},\v{\eta}}$ is obtained by minimizing the loss
\begin{equation}\label{eq:loss_vsk_plus_deltann}
        \begin{split}
    \ell_{X_n}(\v{\theta}, \v{\eta}) & := \frac{1}{n}
    \norm{
    \v{f} - \mK_{\bar{f}_{\v{\theta}}} \v{c}_{\v{\eta}}
    }_2^2 \\
     & = \frac{1}{n}
    \sum_{i=1}^n 
    \left(
    f_i - 
    \sum_{k=1}^n c_k(\v{\eta}) \ \kappa\left((\v{x}_i, \bar{f}_{\v{\theta}}(\v{x}_i))\ , \ (\v{x}_k, \bar{f}_{\v{\theta}}(\v{x}_k))\right)
    \right)^2,
    \end{split}
    \end{equation}
where a global minimum $(\v{\theta}^*,\v{\eta}^*)\in\R^p\times\R^q$ for \eqref{eq:loss_vsk_plus_deltann} is such that $\ell_{X_n}(\v{\theta}^*, \v{\eta}^*) = 0$ (i.e., interpolation satisfied). 

We drive the reader's attention towards the fact that \eqref{eq:loss_vsk_plus_deltann} is defined using the whole training data $\{(\v{x}_i, f_i)\}_{i=1}^n$ and not using random mini-batches sampled from them; indeed, since the main task is the interpolation, we prefer a deterministic training procedure instead of a stochastic one. 

\begin{remark}
    \label{remark1}
    Let $(\v{\theta}^{\rm fin}, \v{\eta}^{\rm fin})$ be the parameters obtained at the end of the training by minimizing \eqref{eq:loss_vsk_plus_deltann}. It is likely that the interpolation conditions corresponding to $(\v{\theta}^{\rm fin}, \v{\eta}^{\rm fin})$ are only approximately satisfied; indeed, stopping criteria are adopted and it is possible that the training stops before reaching a given tolerance. Therefore, it is recommended to build a model based on  $\v{\theta}^{\rm fin}$ and an updated version of $\v{c}$, namely $\v{c}=\mK^{-1}_{\bar{f}_{\v{\theta}^{\rm fin}}} \v{f}$. {In case an iterative solver is used to solve the linear system, $\v{c}_{\v{\eta}^{\rm fin}}$ can be used as a starting guess}.
\end{remark}

Summarizing the procedure, the $\delta$NN-VSK interpolant for a target function $f$ is constructed as in \eqref{eq:VSK_interpol} where the coefficients $\v{c}=(c_1,\ldots ,c_n)^{\intercal}$ are computed as reported in the Remark \ref{remark1} above.

The numerical experiments will show that the scaling function learned via the $\delta$NN-VSK method is, indeed, a nearly-optimal choice and that it closely mimics the target's shape. Therefore, in the next subsection, we propose an alternative way for computing the VSK interpolant, namely the VSK-$f$ approach, that directly learns the target and uses such an approximation as a scaling function.

\subsection{VSKs-$f$: using an approximation of $f$ as scaling function}\label{sec:VSKf}

Another possible approach for defining a proper scaling function can be more focused on the claim stated at the end of Section \ref{sec:theoretic_res}. Specifically, Proposition \ref{prop_lagr} suggests that a good approximation of the target can be used as a ``safe'' scaling function for a VSK interpolant. Therefore, we can renounce to find a nearly-optimal scaling function minimizing a loss like \eqref{eq:loss_vsk_plus_deltann}, preferring to train a $\delta$NN for directly approximating $f$, instead. Nonetheless, we recall that the quantity of data typically used for interpolation (i.e., $n$) rarely is enough for a fine approximation of $f$ via $\delta$NN models; therefore, finding a VSK interpolant of $f$ is still necessary.

In details, given a $\delta$NN $\bar{f}_{\v{\theta}}$, we train it for approximating $f$, using the interpolation data as the training set and via a classic stochastic training procedure. Therefore, the trained model $\bar{f}_{\v{\theta}^{\rm fin}}$ will be an approximation of $f$, and the larger is $n$ (i.e., the amount of interpolation data used for the training) the finer the approximation quality.

In conclusion, the VSK-$f$ method consists in computing the VSK interpolant \eqref{eq:VSK_interpol} of $f$ using $\bar{f}_{\v{\theta}^{\rm fin}}$ as scaling function.

\begin{remark}
    {The main advantage of using  VSKs-$f$ with respect to $\delta$NN-VSKs is that training $\bar{f}_{\v{\theta}}$ for approximating $f$ is easier than training the model $\mathcal{M}_{\v{\theta}, \v{\eta}}$ that learns simultaneously the coefficients of the approximant and the scaling function. Of course, the price of an easier procedure is the risk of using a safe scaling function, which however is not tailored for the considered kernel basis.}
\end{remark}

\section{Numerical experiments}\label{esperimenti}

    {We test the proposed procedure both on synthetic datasets (Subsection \ref{sintetici}) obtained by sampling three selected functions at given sets of quasi-uniform nodes, and on data describing the real-world phenomenon of the acetone's phase transition (Subsection \ref{sec:acetone}). 
    For all the numerical experiments, we have used the same residual $\delta$NN architecture and the same training options. Details are provided in \ref{sec:deltaNN_arch}.
    }

\subsection{Tests with synthetic datasets}\label{sintetici}

    In the numerical experiments that follow we take as test functions the  well-known \emph{Franke's} function
    \begin{equation*}
        \begin{split}
    f_1(x_1,x_2) = & \dfrac{3}{4}{\rm e}^{-[(9x_1-2)^2+(9x_2-2)^2]/4} +\dfrac{3}{4}{\rm e}^{-(9x_1+1)^2/49-(9x_2+1)/10} \\
     & + \dfrac{3}{4}{\rm e}^{-[(9x_1-2)^2+(9x_2-2)^2]/4} +\dfrac{3}{4}{\rm e}^{-(9x_1+1)^2/49-(9x_2+1)/10},
    \end{split}
    \end{equation*}
    which is continuous, being the sum of exponential terms, {and the discontinuous test functions: 
    \begin{align*}
        f_2(x_1,x_2) = \begin{cases}
        -{\rm e}^{(-(x_1-0.5)^2+(x_2-0.5)^2)}, \quad \textrm{if} \quad ||\boldsymbol{x} - (0.5,0.5)||_2^2 \geq 0.08\\
        \sin(x_1)+4\sin(x_2), \quad \textrm{otherwise}, 
             \end{cases}  
      \end{align*}
      and
      \begin{align*}
        f_3(x_1,x_2) = \begin{cases}
        \sin (0.4 \pi (x_1 + x_2)), \quad \textrm{if} \quad x_2 \geq {\rm e}^{x_1},\\
        \sin (0.7 \pi (x_1 + x_2)) - 4, \quad \textrm{if} \quad x_2 < {\rm e}^{x_1}-1,\\
        \sin (\pi (x_1 + x_2)) +4, \quad \textrm{if} \quad \textrm{otherwise}.
        \end{cases}  
    \end{align*}}
    We sample these functions at $n=[27^2,33^2,39^2]$ scattered Halton data in $[0,1]^2$ and we interpolate them with respect to:
    \begin{itemize}
        \item classical Fixed Scale Kernels (FSKs), see \eqref{eq1};
        \item $\delta$NN-VSKs (see Subsection \ref{sec:delta});
        \item VSKs-$f$ (see Subsection \ref{sec:VSKf}).
    \end{itemize}
    {The dataset size considered here is rather small, indeed, in order to reconstruct the target when the dataset is large enough, one may use local and greedy methods \cite{cavorettoEfficient, Santin23} or directly the $\delta$NN.}

    For all the VSK interpolants, we consider a scalar scaling function $\bar{f}_{\v{\theta}}$; i.e., $\bar{f}_{\v{\theta}}:\R^2\rightarrow\R$. 
    
    All the interpolants are evaluated on a $\sqrt{N}\times \sqrt{N}$ regular grid of evaluation points $\boldsymbol{\xi}_k$, $k=1,\ldots ,N$, with $\sqrt{N}=100$. In particular, we compute the following performance indicators for the interpolation accuracy of a given approximant $P$:
    \begin{itemize}
        \item Mean Absolute Error (MAE):
        $$
        \mathrm{MAE}(P)=\frac{1}{N}\sum_{k=1}^{N}|f(\boldsymbol{\xi}_k) - P(\boldsymbol{\xi}_k)|;
        $$
        \item Mean Squared Error (MSE):
        $$
        \mathrm{MSE}(P)=\frac{1}{N}\sum_{k=1}^{N}(f(\boldsymbol{\xi}_k) - P(\boldsymbol{\xi}_k))^2.
        $$
    \end{itemize}   
    
    Moreover, we also compute the Structural Similarity Index Measure (SSIM) \cite{SSIM_paper_2004} with respect to the greyscale image $I_f$ generated by the evaluations of $f$ on the $\sqrt{N}\times \sqrt{N}$ grid and the greyscale image $I_P$ generated by the evaluations of the interpolant $P$ on the same grid. The value $\mathrm{SSIM}(I_f,I_P)$ is in $[-1,1]$, where $1$ denotes perfect similarity, $0$ denotes no similarity, and $-1$ denotes perfect anti-correlation. {The SSIM index is an accuracy score popular in many imaging applications, such as image processing, super-resolution, image resizing, image rotation and registration, and whose convergence in the context of image interpolation has been recently studied in \cite{MarchettiSSIM}}. This index is computed via the \emph{tf.image.ssim} function of Tensorflow \cite{SSIM_TF}. For more details about SSIM refer to \cite{SSIM_paper_2004,SSIM_TF}.
        
    The kernel used to approximate the function $f_1$ (independently of the method used) is the Gaussian $C^{\infty}$ function:
    $$
    \phi_{\varepsilon}(r)=\rm{e}^{-\varepsilon^2r^2}\,,
    $$ 
    while for {$f_2$ and $f_3$} we use the Mat\'ern $C^2$ kernel
    $$
    \phi_{\varepsilon}(r)=\rm{e}^{-\varepsilon r}(1+\varepsilon r)\,.
    $$ 
    
    The results are reported in Table \ref{tab:tabella_1}. All results are reproducible as the Python code is freely available at 
    \begin{center}
        \url{https://github.com/Fra0013To/VSKlearning/tree/paper2024}.
    \end{center}
    {In Figures \ref{fig:f1_topview_comparison}--\ref{fig:f3_surface_errors_horizontal}, we depict all the obtained surfaces. In particular, we plot both their top view, which is essential to clearly see the improvements in the discontinuous case, and their 3D counterpart false colored with the absolute error.}
    
    {In all the considered test cases, we observe (see Table \ref{tab:tabella_1}) that the $\delta$NN-VSK and VSK-$f$ methods behave better than FSKs. Precisely, for the continuous test function $f_1$, the FSKs already give good approximations, while in the discontinuous cases the $\delta$NN-VSK and VSK-$f$ methods provide a consistent improvement on the approximation accuracy (see Table \ref{tab:tabella_1} and Figures \ref{fig:f1_topview_comparison}-\ref{fig:f3_surface_errors_horizontal}); moreover, in the latter case, the improvement in terms of SSIM (refer to Table \ref{tab:tabella_1}) of the VSKs-based methods is rather impressive. 
    Such results are due to the fact that when data are smooth, VSKs might only improve already quite good approximation scores returned by FSKs. On the other hand, in the discontinuous cases, the hardest task is to identify and reconstruct the faults; therefore, any discontinuous scaling function sharing the same discontinuities of the target provides a positive contribution.}
    
\begin{table}[htb]
\centering
\resizebox{1.\textwidth}{!}{
\begin{tabular}{|c|ccc|ccc|ccc|ccc|}
\hline 
& \multicolumn{3}{c|}{Interpolation} & \multicolumn{3}{c|}{MAE} & \multicolumn{3}{c|}{MSE} & \multicolumn{3}{c|}{SSIM} \\
\hline 
& $n$ & Kernel & $\varepsilon$ & FSKs & $\delta$NN-VSKs & VSKs-$f$ & FSKs & $\delta$NN-VSKs & VSKs-$f$ & FSKs & $\delta$NN-VSKs & VSKs-$f$  \\
\hline 
\multirow{3}{*}{$f_1$} 
& 729 & Gauss. & 0.6 & 4.99\rm{e}-2 & 8.98\rm{e}-3 & 3.57\rm{e}-3 & 4.24\rm{e}-3 & 1.42\rm{e}-4 & 2.43\rm{e}-5 & 
0.8712 & 0.9820 & 0.9943  \\
& 1089 & Gauss. & 1.2 & {2.38\rm{e}-2} & 3.48\rm{e}-3 & 2.11\rm{e}-3 & 9.88\rm{e}-4 & 2.33\rm{e}-5 & 8.25\rm{e}-6 &
0.9533 & 0.9969 & 0.9980 \\
& 1521 & Gauss. & 4.8 & 3.60\rm{e}-4 & 3.60\rm{e}-4 & 2.47\rm{e}-4 & 3.84\rm{e}-7 & 3.83\rm{e}-7 & 2.70\rm{e}-7 &
0.9997 & 0.9997 & 0.9999 \\
\hline
\multirow{3}{*}{$f_2$} 
& 729 & Mat.$C^2$ & 0.06 & 1.97\rm{e}-1 & 2.09\rm{e}-2 & 1.65\rm{e}-2 &  5.76\rm{e}-2 & 3.72\rm{e}-3 & 3.84\rm{e}-3 &
0.5883 & 0.9074 & 0.9253 \\
& 1089 &  Mat.$C^2$ & 0.12 & {1.45\rm{e} -1} & 8.06\rm{e}-3 & 1.29\rm{e}-2 & 3.37\rm{e}-2 & 2.12\rm{e}-3 & 2.18\rm{e}-3 &
0.5885 & 0.9556 & 0.9486 \\
& 1521 & Mat.$C^2$ & 0.48 & 6.65\rm{e}-2 & 5.68\rm{e}-3 & 8.37\rm{e}-3 & 1.33\rm{e}-2 & 1.89\rm{e}-3 & 2.37\rm{e}-3 &
0.6453 & 0.9696 & 0.9559 \\
\hline
\multirow{3}{*}{$f_3$} 
& 729 & Mat.$C^2$ & 0.06 & 2.09\rm{e}-1 & 4.30\rm{e}-2 & 1.40\rm{e}-2 &  5.84\rm{e}-2 & 3.82\rm{e}-3 & 2.42\rm{e}-3 &
0.6594 & 0.7928 & 0.9495 \\
& 1089 & Mat.$C^2$ & 0.12 & 1.34\rm{e}-1& 3.65\rm{e}-2 & 1.98\rm{e}-2 &  3.08\rm{e}-2 & 3.58\rm{e}-3 & 2.71\rm{e}-3 &
0.6411 & 0.8993 & 0.9312 \\
& 1521 & Mat.$C^2$ & 0.48 & 6.28\rm{e}-2 & 1.00\rm{e}-2 & 6.48\rm{e}-3 &  1.20\rm{e}-2 & 1.66\rm{e}-3 & 1.19\rm{e}-3 &
0.7242 & 0.9657 & 0.9760 \\
\hline 
\end{tabular}
}
\caption{{MAEs, MSEs, and SSIMs for the test functions $f_1$, $f_2$, and $f_3$. 
}
}
\label{tab:tabella_1}
\end{table}
    
{For all the cases, we recall that no information is provided about the continuity or discontinuity of the target, nor is information provided for selecting the scaling function. Therefore, the continuous case is a rather interesting test case, because it shows the flexibility of the $\delta$NN architecture in learning also continuous targets. In other words, if one has no information about the smoothness of data, then the $\delta$NN-VSK approach is a practical and robust choice for approximating the target.

Moreover, concerning the scaling function learned via $\delta$NNs, we observe that for both the continuous test ($f_1$) and all the other discontinuous tests, the $\delta$NN-VSK method learns nearly-optimal scaling functions that mimic the targets, while the VSK-$f$ method supports the theoretical claims of the safety in selecting a scaling function that not only mimics but approximates the target; refer to the bottom rows of Figures \ref{fig:f1_topview_comparison}, \ref{fig:f2_topview_comparison}, and \ref{fig:f3_topview_comparison}}.
    
{Summarizing, from the analysis of the interpolation results, we observe the following:}

\begin{itemize}
    
    \item Our VSKs-based methods provide reliable approximations {both for the continuous case and for the discontinuous test functions (see Table \ref{tab:tabella_1} and Figures \ref{fig:f1_surface_errors_horizontal}--\ref{fig:f3_surface_errors_horizontal})};
    
    \item The $\delta$NN-VSK method proves to be able to learn a nearly-optimal scaling function automatically, and as a numerical confirmation of our claims, the identified scaling functions $\bar{f}_{\v{\theta}^{\rm fin}}$ are such that they \emph{mimic} the target $f$. This mimicking behavior can be detailed {(e.g., see Figure \ref{fig:f1_topview_comparison}) or ``generic'' (e.g., see Figures \ref{fig:f2_topview_comparison}--\ref{fig:f3_topview_comparison});} 
    
    \item The results obtained via the VSK-$f$ method support our theoretical claim reported at the end of Section \ref{sec:theoretic_res}; i.e., using a scaling function $\bar{f}_{\v{\theta}^{\rm fin}}\approx f$ is a safe choice that guarantees a good interpolation accuracy, even if not necessarily the \emph{optimal one} (see Table \ref{tab:tabella_1}).
\end{itemize}

Concerning the second item above,  it is worth remarking that the function ${\bar{f}}:=\bar{f}_{\v{\theta}^{\rm fin}}$ of the $\delta$NN-VSK method
    is such that $\bar{f} \approx \gamma f$, where $\gamma: \mathbb{R}^d \longrightarrow \mathbb{R}$ is a continuous function. This is coherent with our implementation of the $\delta$NN and its loss \eqref{eq:loss_vsk_plus_deltann}. Let us write the VSKs with their explicit dependence on the shape parameter and on the radial distance, i.e., for $\boldsymbol{x},\boldsymbol{z}\in \Omega$:
        \begin{equation*} 
    \begin{split}
    \kappa_{\gamma {f}}(\boldsymbol{x},\boldsymbol{z}) & = \phi\left(\varepsilon\sqrt{\sum_{k=1}^d ({x}^{(k)}-{z}^{(k)})^2+((\gamma {f})(\boldsymbol{x})- (\gamma {f})(\boldsymbol{z}))^2}\right)\\
     & = \phi\left(\sqrt{\sum_{k=1}^d\varepsilon^2({x}^{(k)}-{z}^{(k)})^2+\varepsilon^2 (\gamma(\boldsymbol{x}) {f}(\boldsymbol{x})-{\gamma(\boldsymbol{z}) {f}(\boldsymbol{z}))^2}}\right). 
    \end{split}
    \end{equation*}
    {Then, the fact that the function $\bar{f}$ mimics the shape of the function $f$ but sometimes possibly out of scale is a consequence of the fact that the $\delta$NN itself tries to \emph{adjust}, via the function $\gamma$, the fixed shape parameter $\varepsilon$. In principle, we could optimize the shape parameter too, indeed methods for optimizing $\varepsilon$ could be integrated with our scheme \cite{driscoll2002interpolation,ling2022stochastic,cavoretto2021,Wenzel}; nevertheless, this is beyond the scope of the current paper as here the focus is on learning $\bar{f}$; such feature can be addressed in future work.}
    However, the most relevant result of these experiments is that we always observe some sort of similarities between the targets and the recovered scaling functions (coherently with Proposition \ref{prop_lagr}) and that the proposed VSK-based methods work properly for automatically identifying nearly-optimal/safe scaling functions.

\begin{figure}[htbp!]
    \centering
    \subcaptionbox{Target}{
    \includegraphics[trim={2.3cm 1.6cm 2.45cm 1.9cm},clip,width=0.38\textwidth]{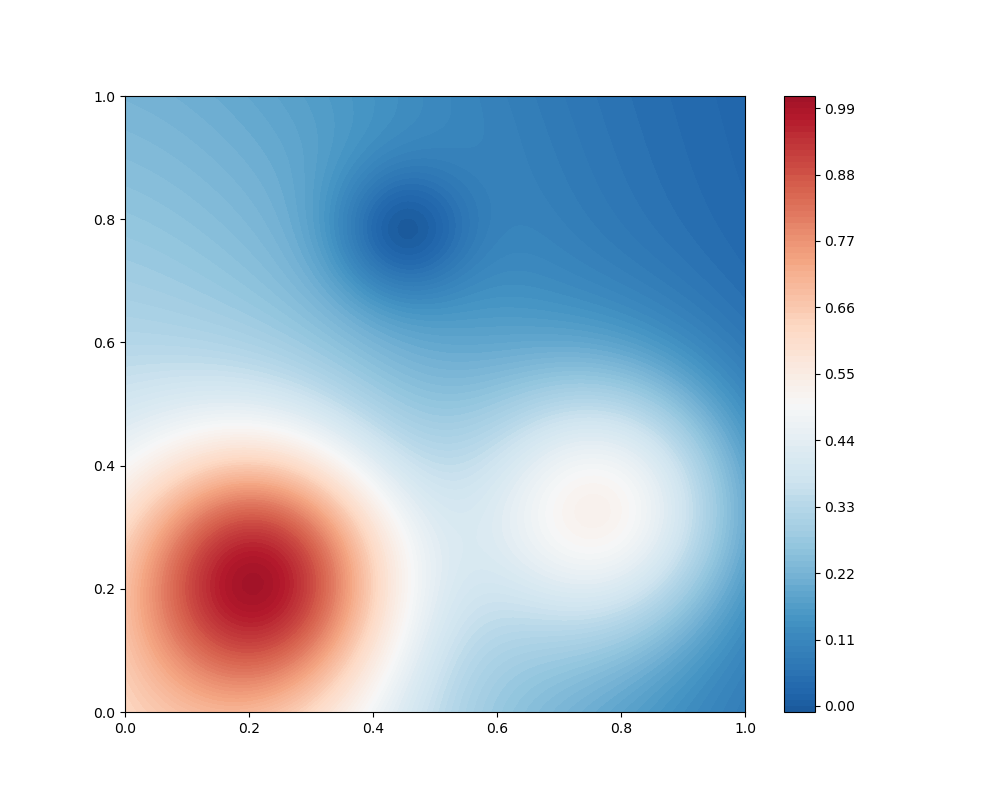}
    }
    \subcaptionbox{FSKs}{
    \includegraphics[trim={2.3cm 1.6cm 2.45cm 1.9cm},clip,width=0.38\textwidth]{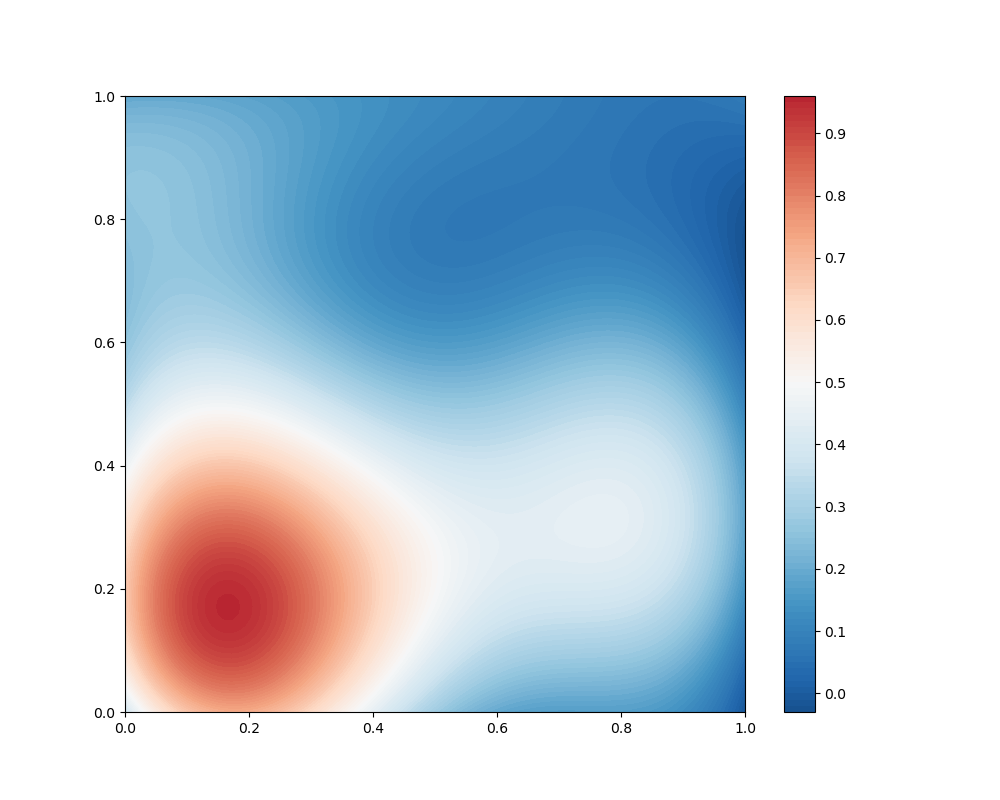}
    }
    \\
    \subcaptionbox{$\delta$NN-VSKs}{
    \includegraphics[trim={2.3cm 1.6cm 2.45cm 1.9cm},clip,width=0.38\textwidth]{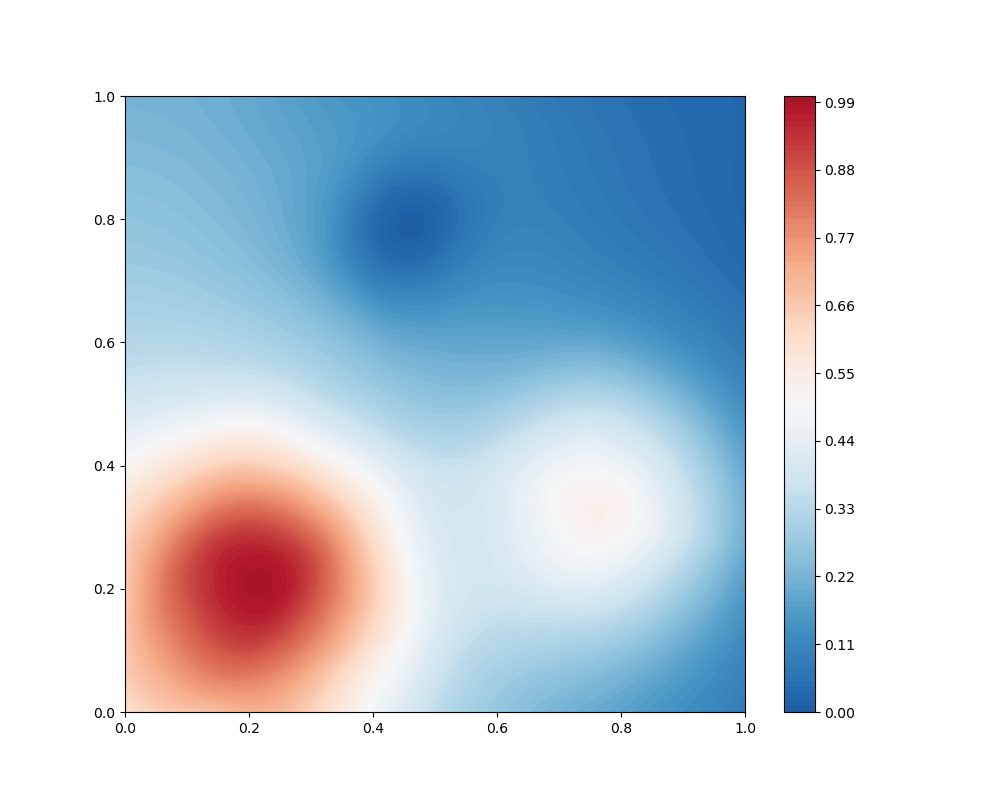}
    }
    \subcaptionbox{VSKs-$f$}{
    \includegraphics[trim={2.3cm 1.6cm 2.45cm 1.9cm},clip,width=0.38\textwidth]{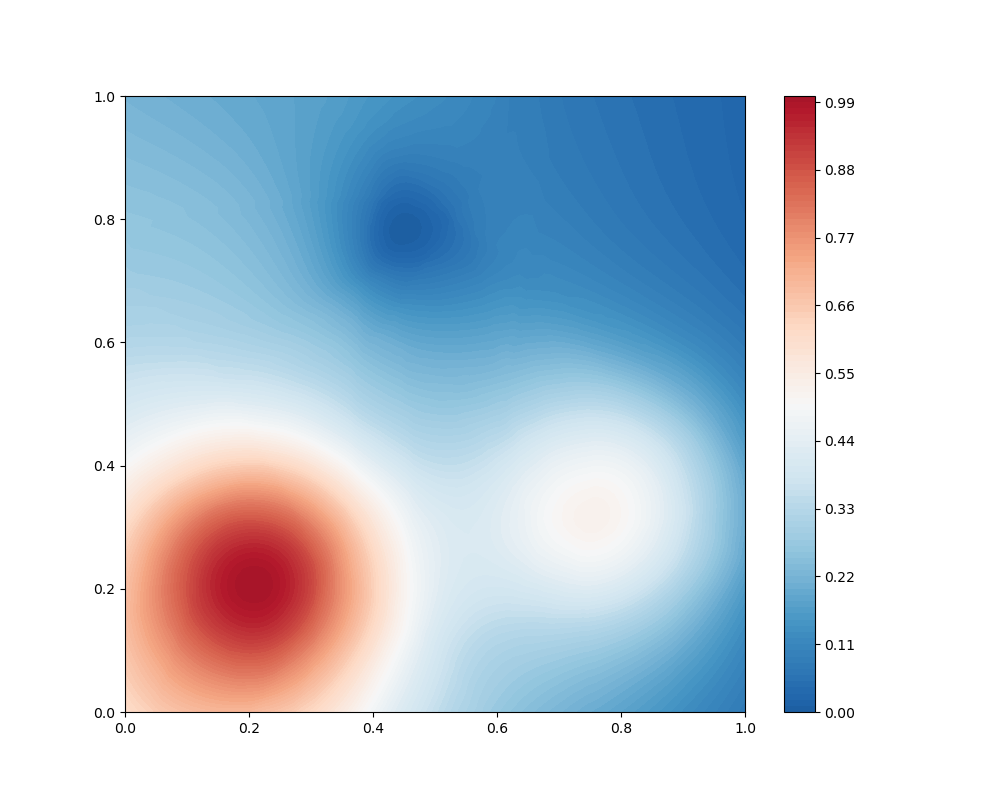}
    }
    \\
    \subcaptionbox{$\delta$NN-VSKs,\\ scaling func.}{
    \includegraphics[trim={2.3cm 1.6cm 2.45cm 1.9cm},clip,width=0.38\textwidth]{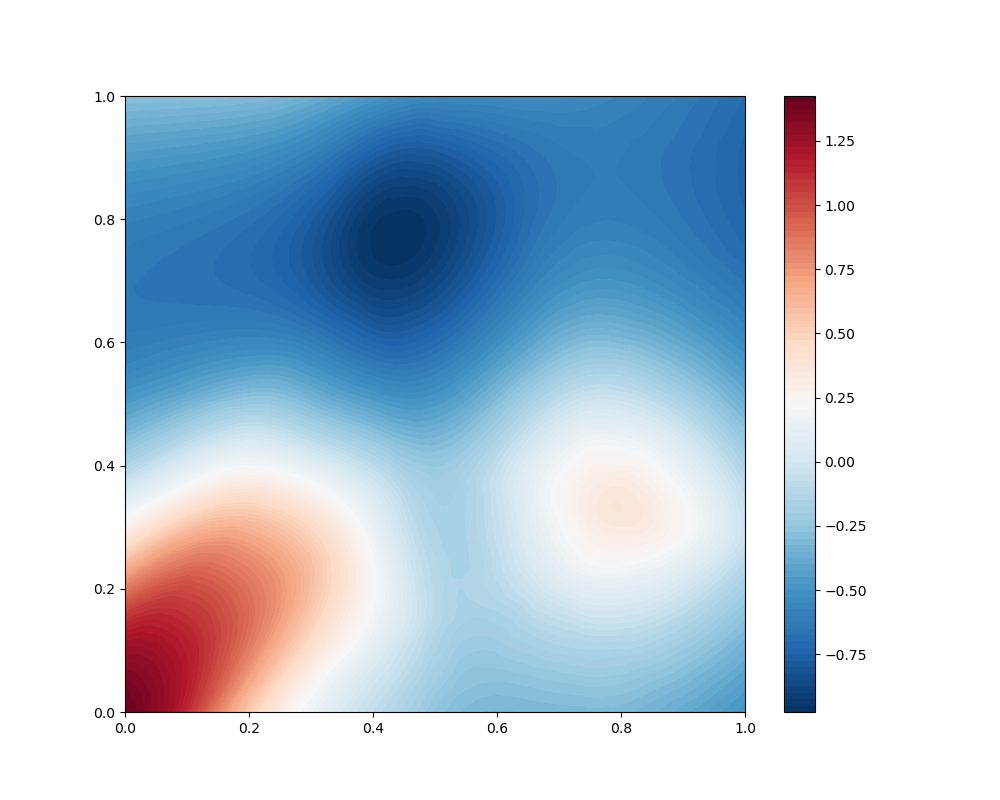}
    }
    \subcaptionbox{VSKs-$f$,\\ scaling func.}{
    \includegraphics[trim={2.3cm 1.6cm 2.45cm 1.9cm},clip,width=0.38\textwidth]{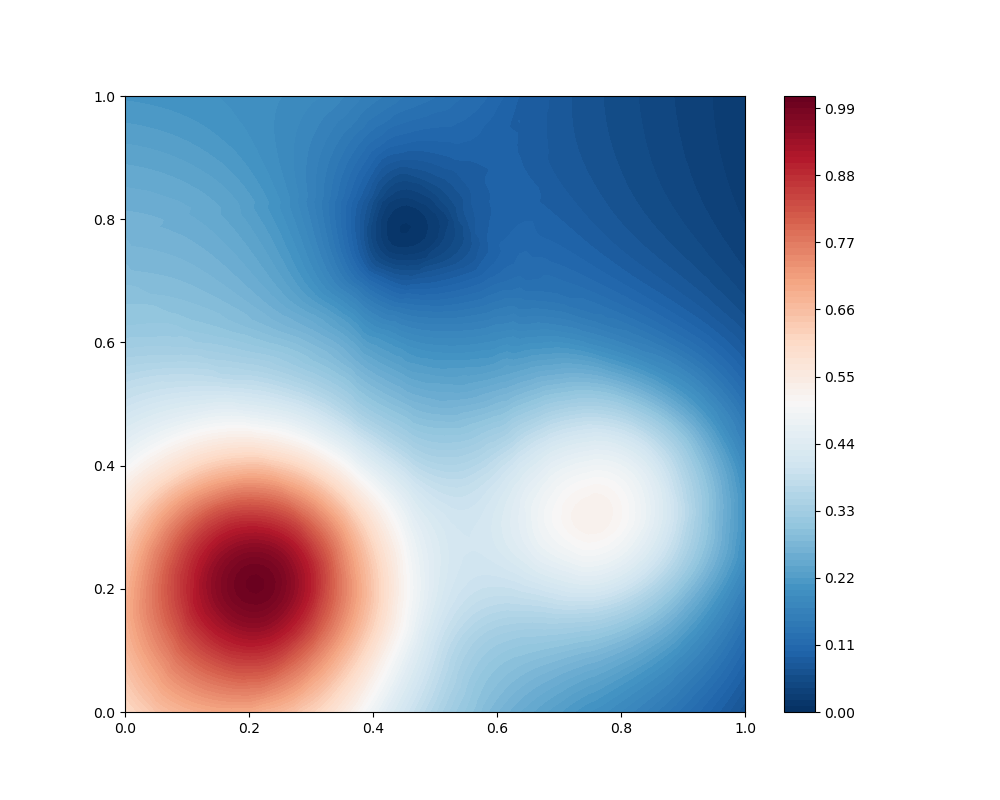}
    }
    \caption{{Interpolation results for $f_1$. Case $n=33^2=1089$ interpolation points, Gaussian Kernel. Shared color bars for subfigures ($a$)-($d$); custom color bars for the scaling functions (subfigures ($e$) and ($f$)).}}
    \label{fig:f1_topview_comparison}
\end{figure}

\begin{figure}[htbp!]
    \centering
    \subcaptionbox{Target}{
    \includegraphics[trim={2.3cm 1.6cm 2.45cm 1.9cm},clip,width=0.38\textwidth]{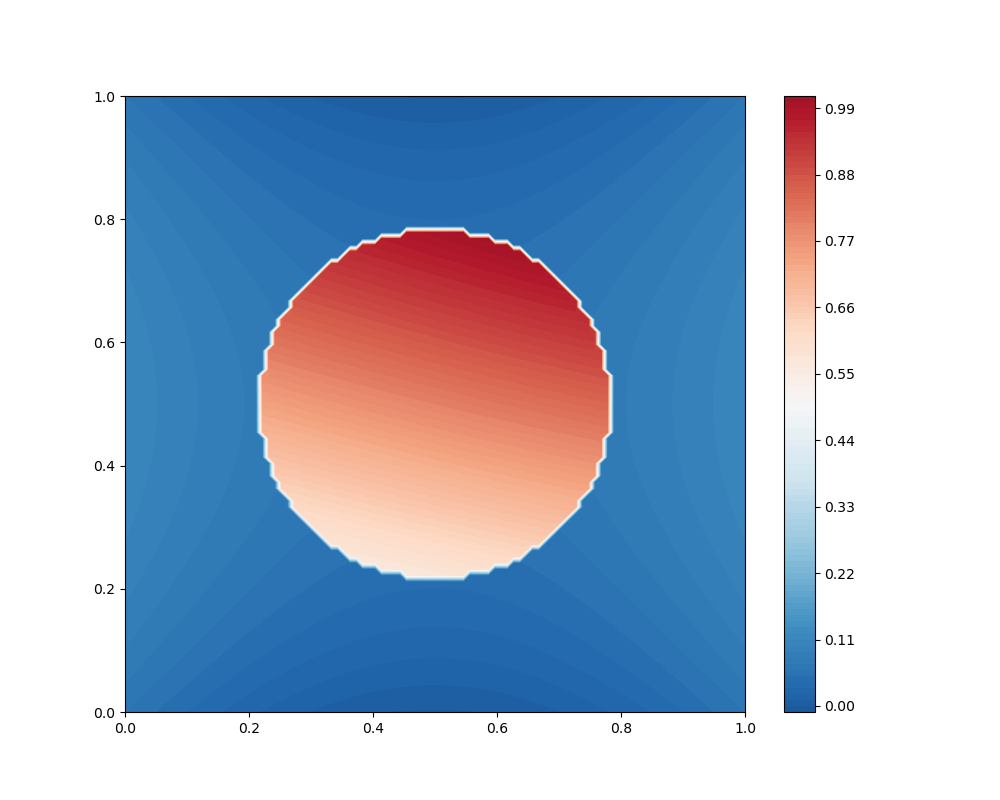}
    }
    \subcaptionbox{FSKs}{
    \includegraphics[trim={2.3cm 1.6cm 2.45cm 1.9cm},clip,width=0.38\textwidth]{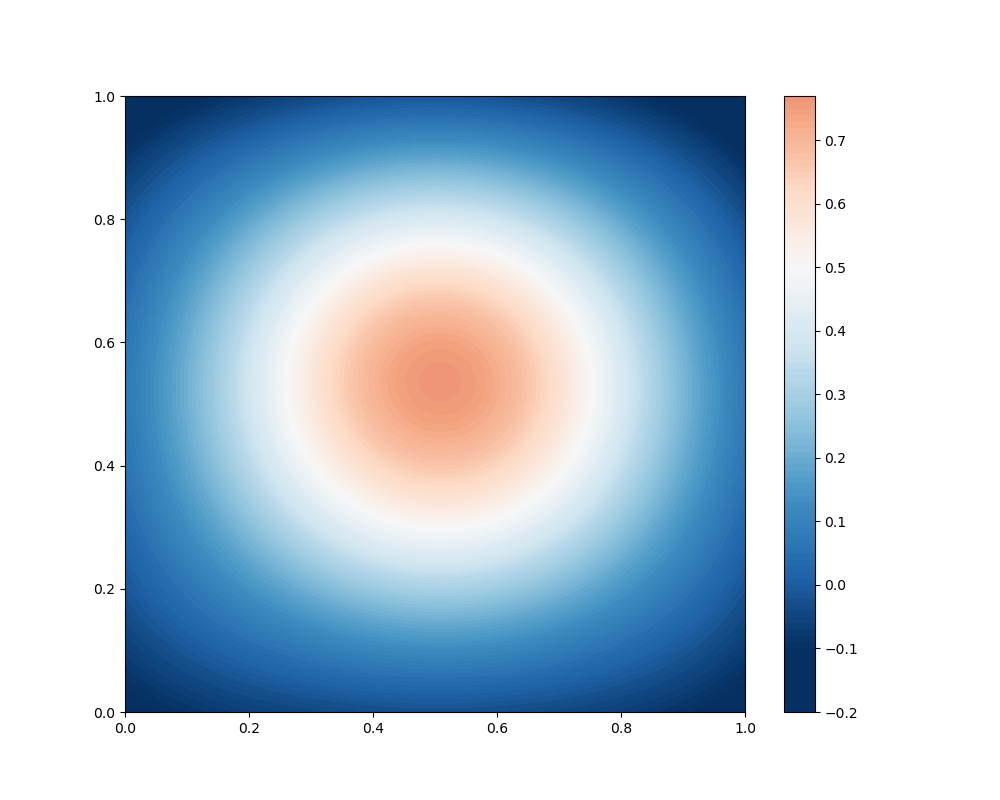}
    }
    \\
    \subcaptionbox{$\delta$NN-VSKs}{
    \includegraphics[trim={2.3cm 1.6cm 2.45cm 1.9cm},clip,width=0.38\textwidth]{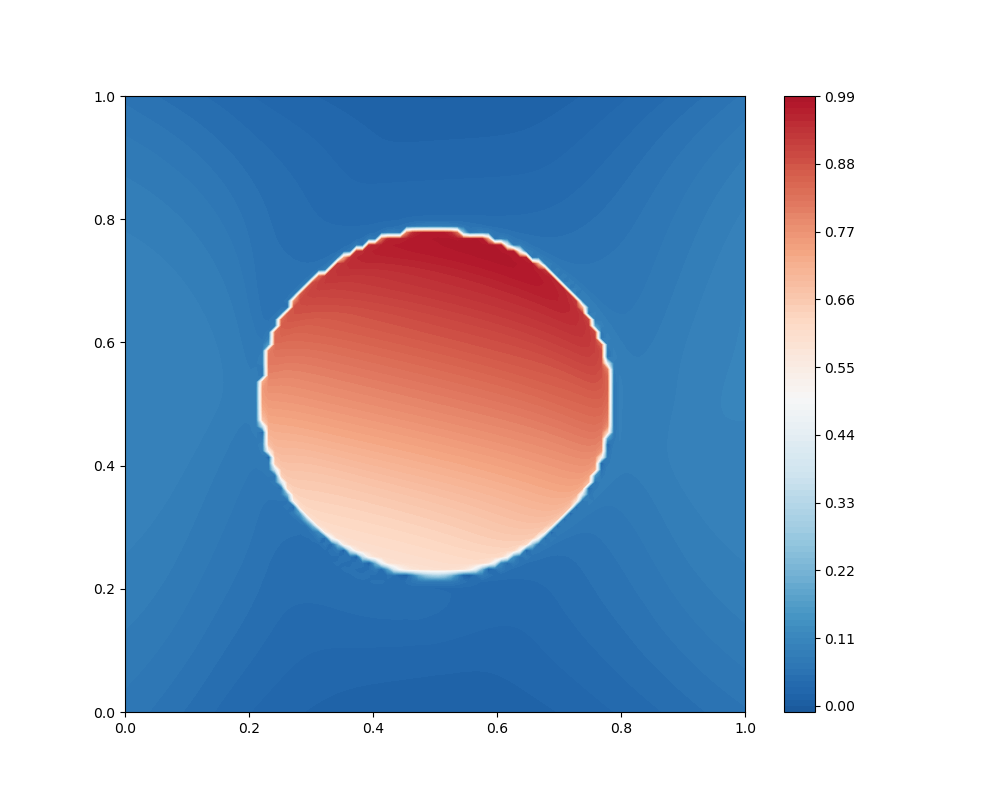}
    }
    \subcaptionbox{VSKs-$f$}{
    \includegraphics[trim={2.3cm 1.6cm 2.45cm 1.9cm},clip,width=0.38\textwidth]{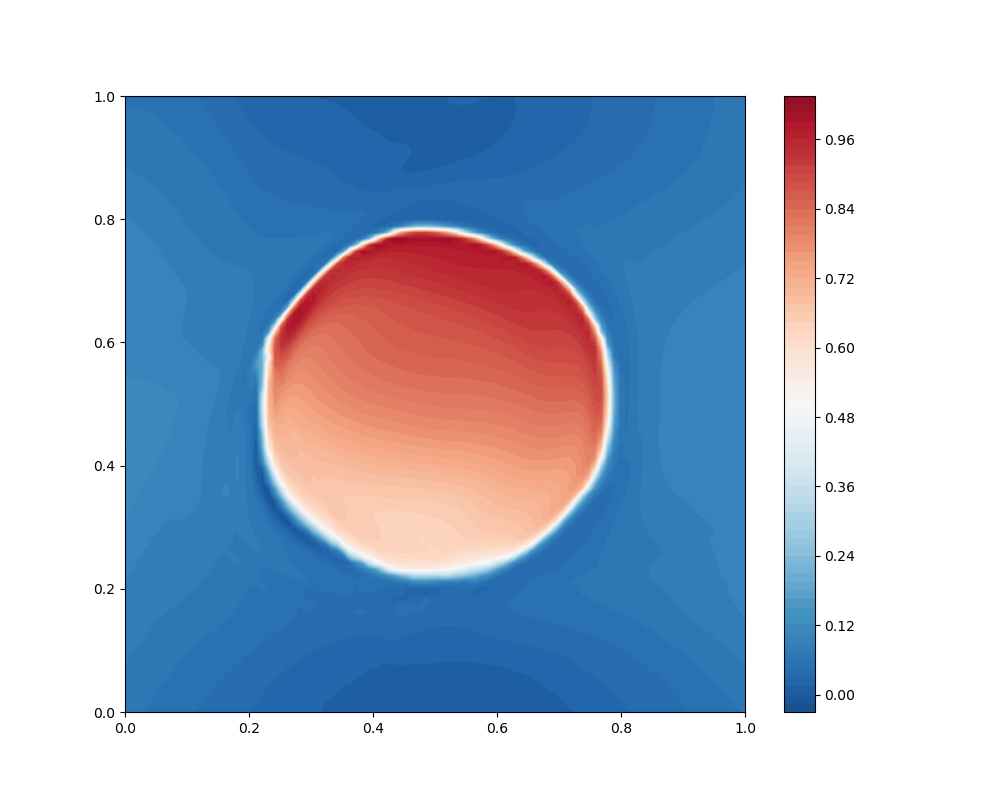}
    }
    \\
    \subcaptionbox{$\delta$NN-VSKs,\\ scaling func.}{
    \includegraphics[trim={2.3cm 1.6cm 2.45cm 1.9cm},clip,width=0.38\textwidth]{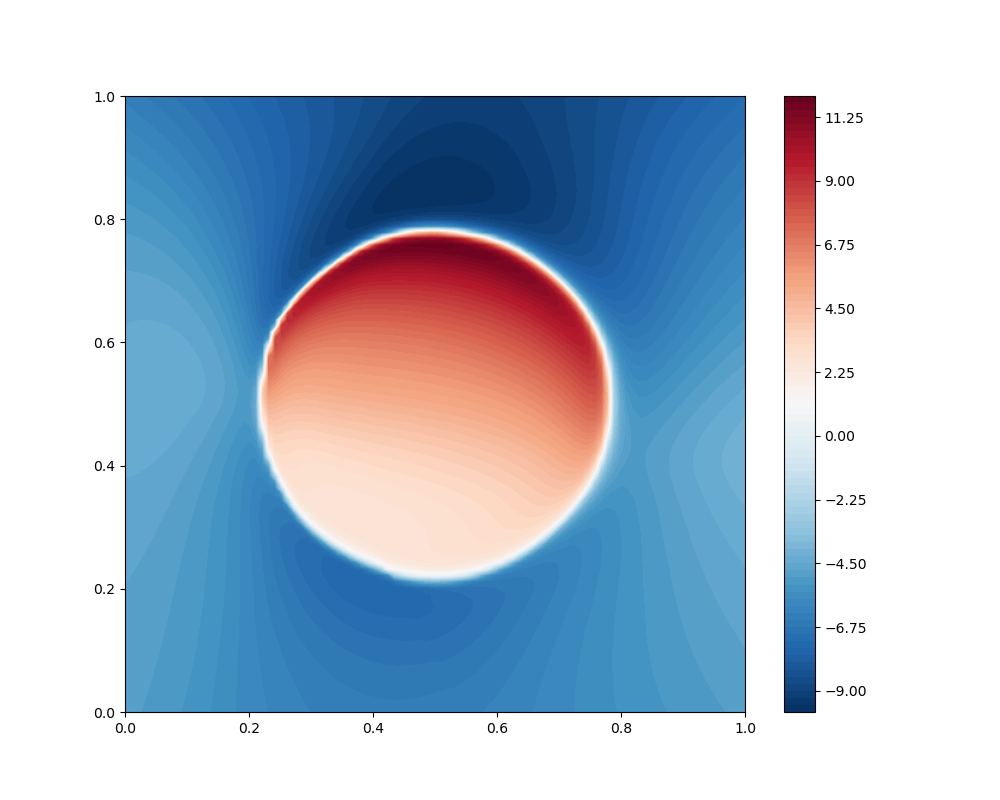}
    }
    \subcaptionbox{VSKs-$f$,\\ scaling func.}{
    \includegraphics[trim={2.3cm 1.6cm 2.45cm 1.9cm},clip,width=0.38\textwidth]{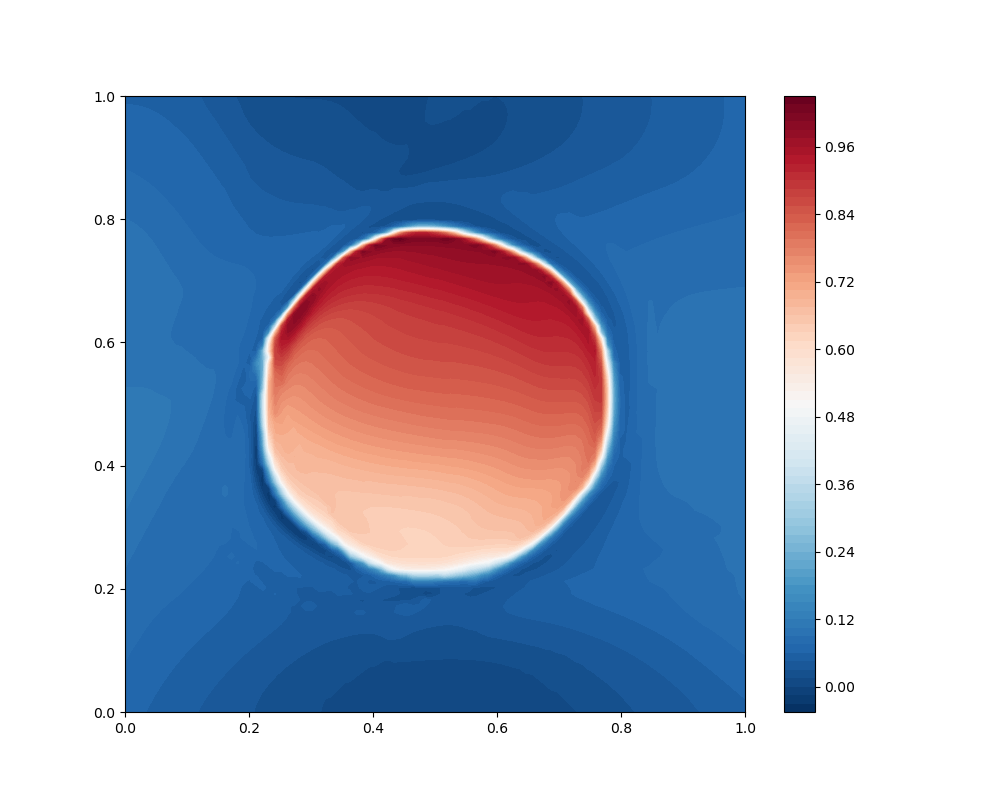}
    }
    \caption{{Interpolation results for $f_2$. Case $n=33^2=1089$ interpolation points, Mat\'ern-$C^2$ Kernel. Shared color bars for subfigures ($a$)-($d$); custom color bars for the scaling functions (subfigures ($e$) and ($f$)).}}
    \label{fig:f2_topview_comparison}
\end{figure}

\begin{figure}[htbp!]
    \centering
    \subcaptionbox{Target}{
    \includegraphics[trim={2.3cm 1.6cm 2.45cm 1.9cm},clip,width=0.38\textwidth]{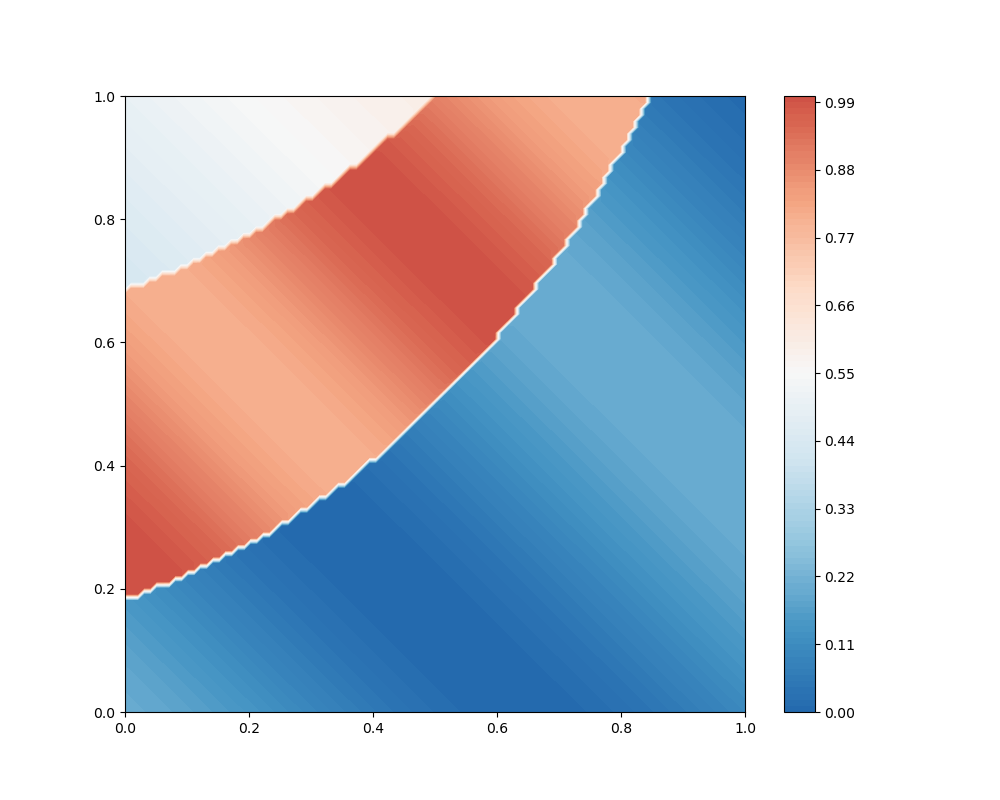}
    }
    \subcaptionbox{FSKs}{
    \includegraphics[trim={2.3cm 1.6cm 2.45cm 1.9cm},clip,width=0.38\textwidth]{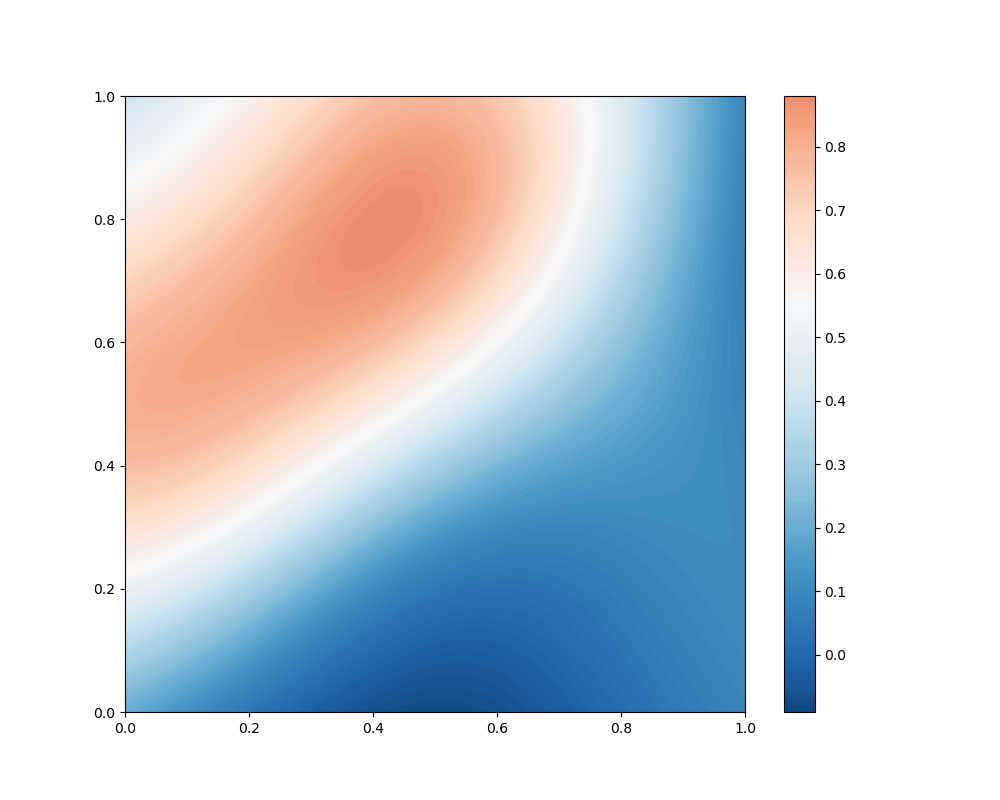}
    }
    \\
    \subcaptionbox{$\delta$NN-VSKs}{
    \includegraphics[trim={2.3cm 1.6cm 2.45cm 1.9cm},clip,width=0.38\textwidth]{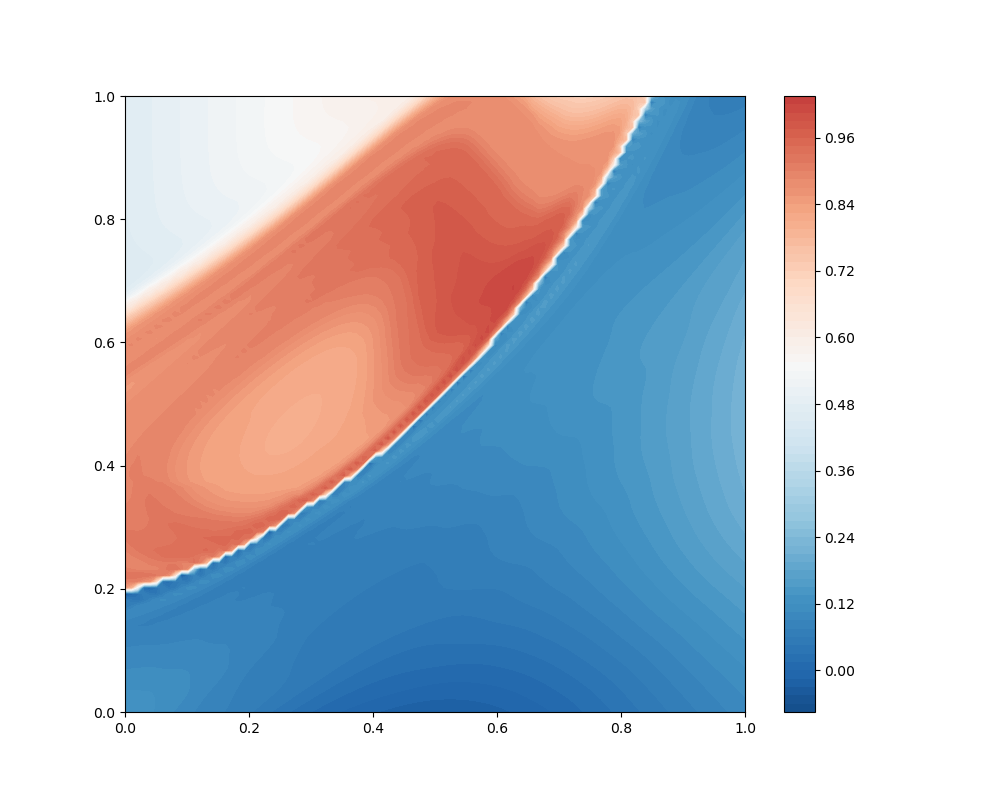}
    }
    \subcaptionbox{VSKs-$f$}{
    \includegraphics[trim={2.3cm 1.6cm 2.45cm 1.9cm},clip,width=0.38\textwidth]{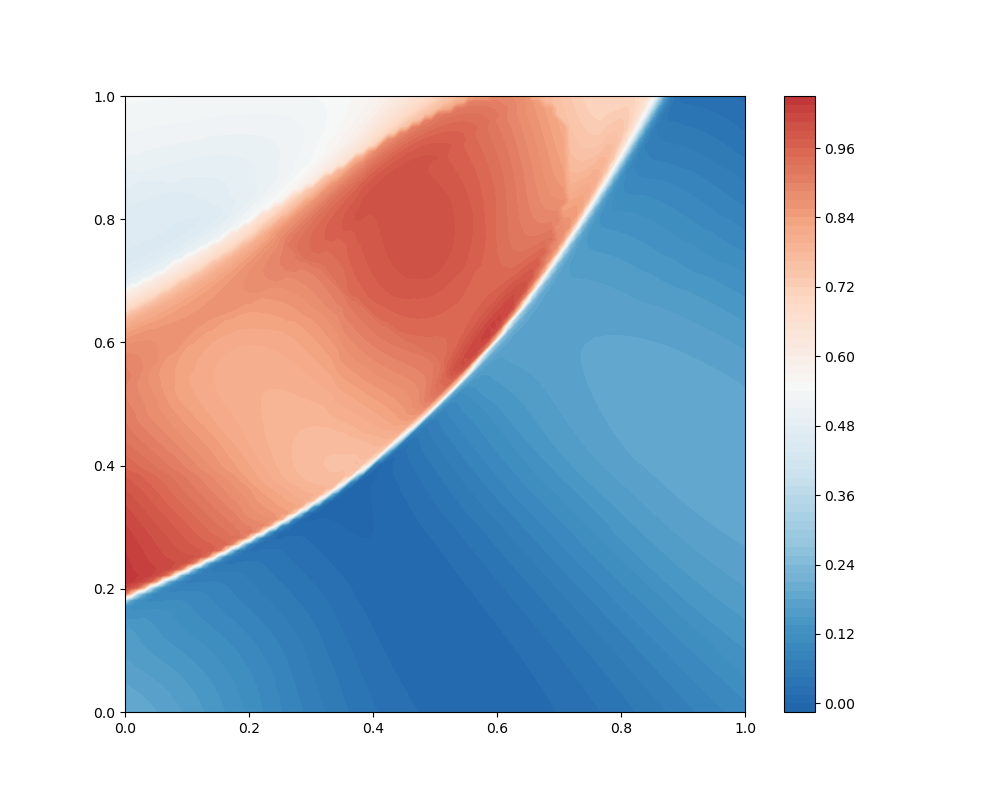}
    }
    \\
    \subcaptionbox{$\delta$NN-VSKs,\\ scaling func.}{
    \includegraphics[trim={2.3cm 1.6cm 2.45cm 1.9cm},clip,width=0.38\textwidth]{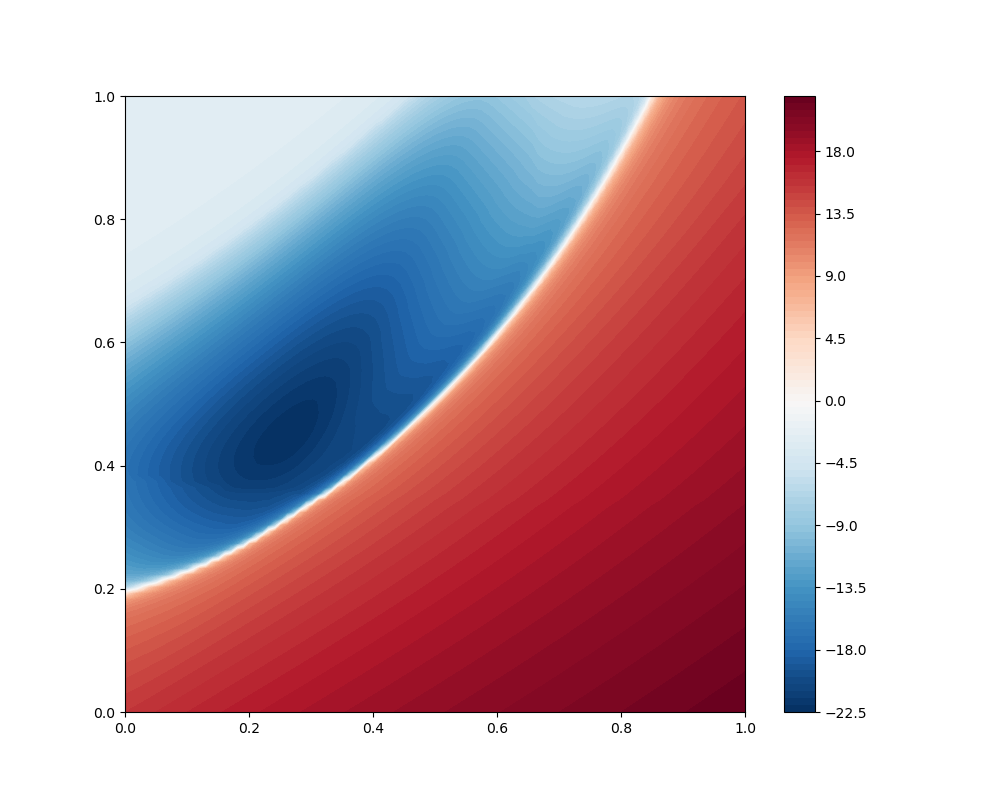}
    }
    \subcaptionbox{VSKs-$f$,\\ scaling func.}{
    \includegraphics[trim={2.3cm 1.6cm 2.45cm 1.9cm},clip,width=0.38\textwidth]{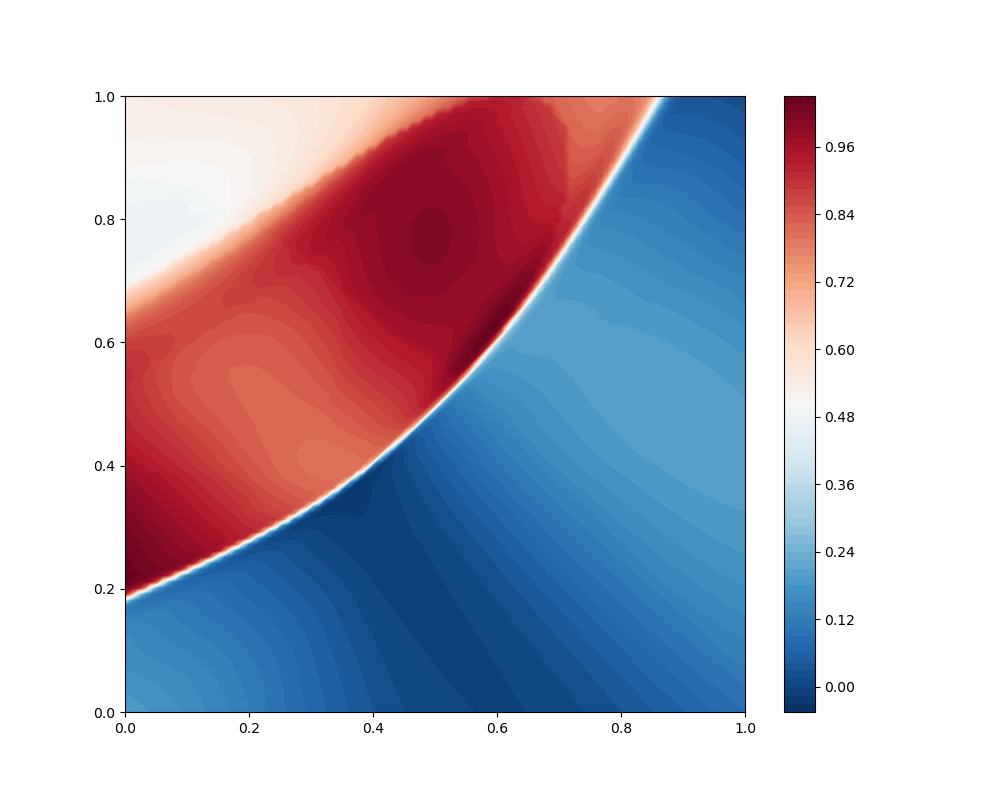}
    }
    \caption{{Interpolation results for $f_3$. Case $n=33^2=1089$ interpolation points, Mat\'ern-$C^2$ Kernel. Shared color bars for subfigures ($a$)-($d$); custom color bars for the scaling functions (subfigures ($e$) and ($f$))}.}
    \label{fig:f3_topview_comparison}
\end{figure}

\begin{figure}[htbp!]
     \centering
     \subcaptionbox{Target}{
     \includegraphics[trim={2.3cm 1.6cm 2.45cm 1.9cm},clip,width=0.38\textwidth]{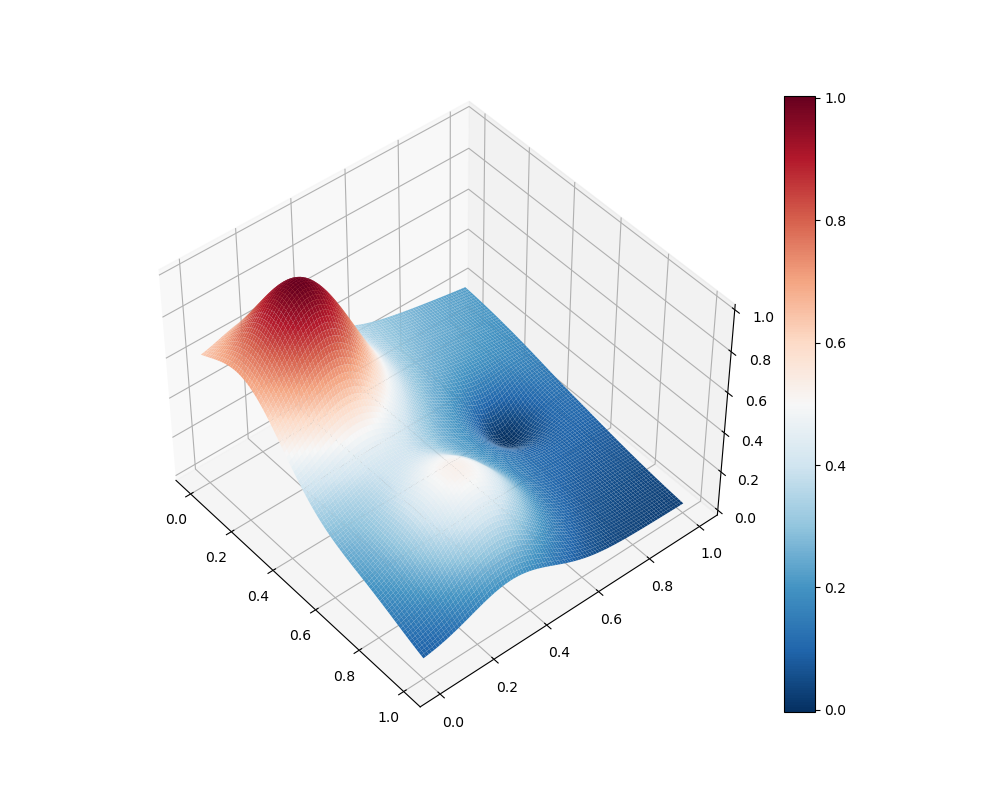}
     }
     \\
     \subcaptionbox{FSKs}{
     \includegraphics[trim={2.3cm 1.6cm 5.55cm 1.9cm},clip,width=0.28\textwidth]{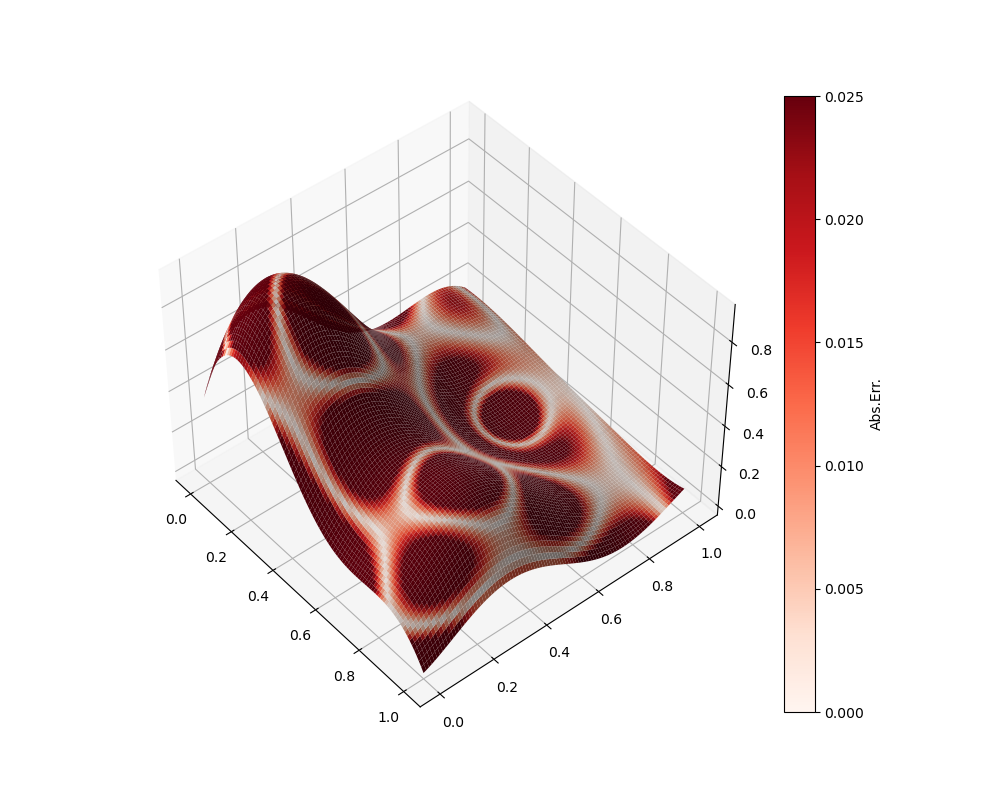}
     }
     \subcaptionbox{$\delta$NN-VSKs}{
     \includegraphics[trim={2.3cm 1.6cm 5.55cm 1.9cm},clip,width=0.28\textwidth]{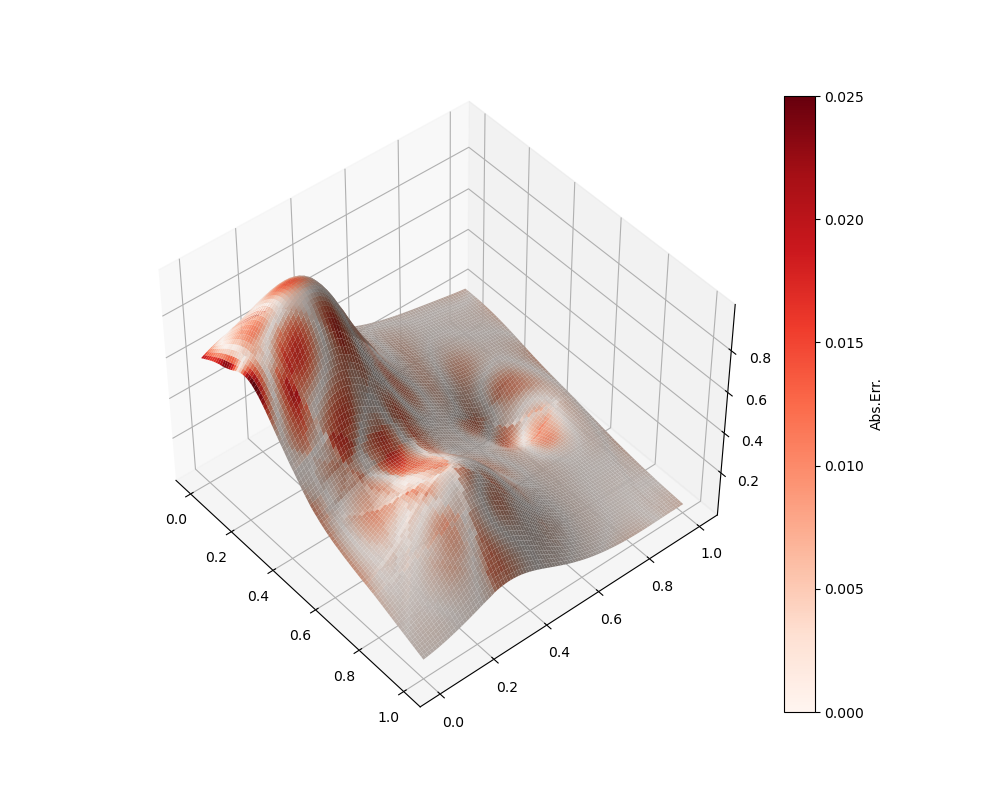}
     }
     \subcaptionbox{VSKs-$f$}{
     \includegraphics[trim={2.3cm 1.6cm 2.45cm 1.9cm},clip,width=0.325\textwidth]{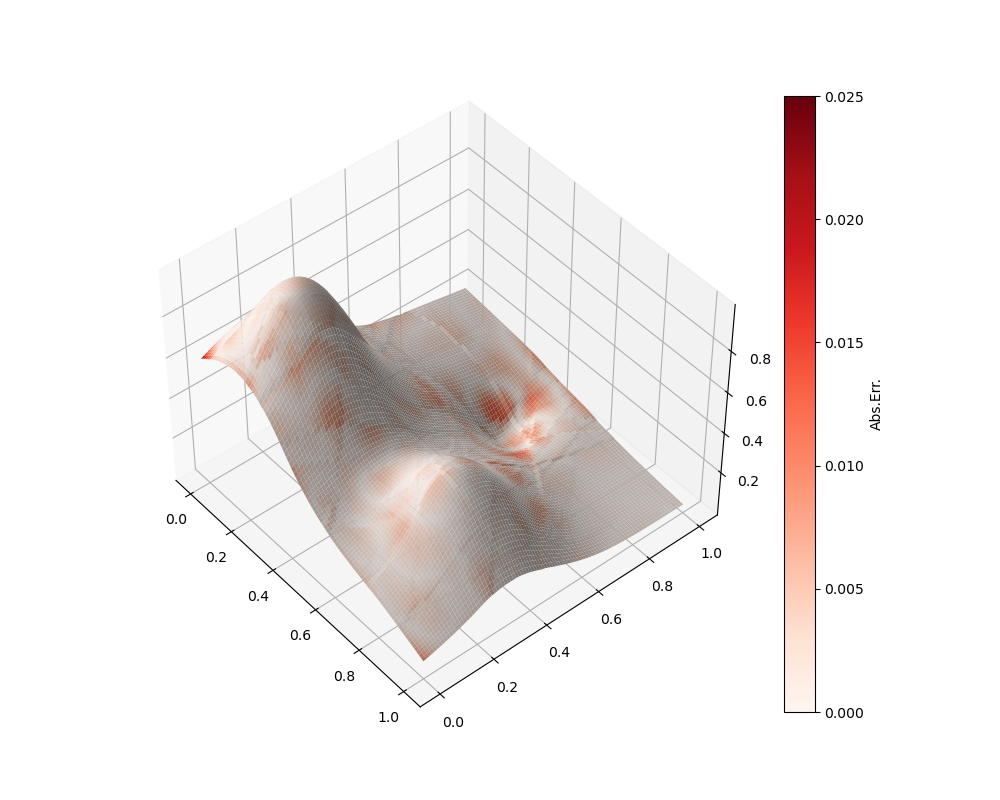}
     }
     \\
     \subcaptionbox{FSKs}{
     \includegraphics[trim={2.3cm 1.6cm 5.55cm 1.9cm},clip,width=0.28\textwidth]{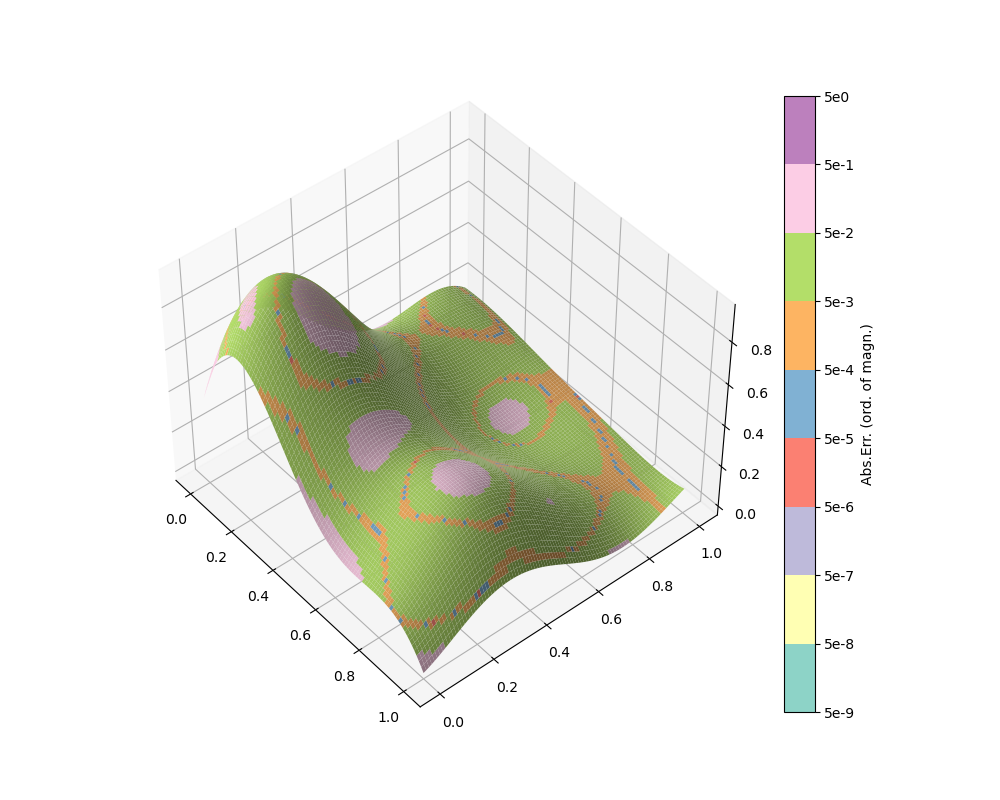}
     }
     \subcaptionbox{$\delta$NN-VSKs}{
     \includegraphics[trim={2.3cm 1.6cm 5.55cm 1.9cm},clip,width=0.28\textwidth]{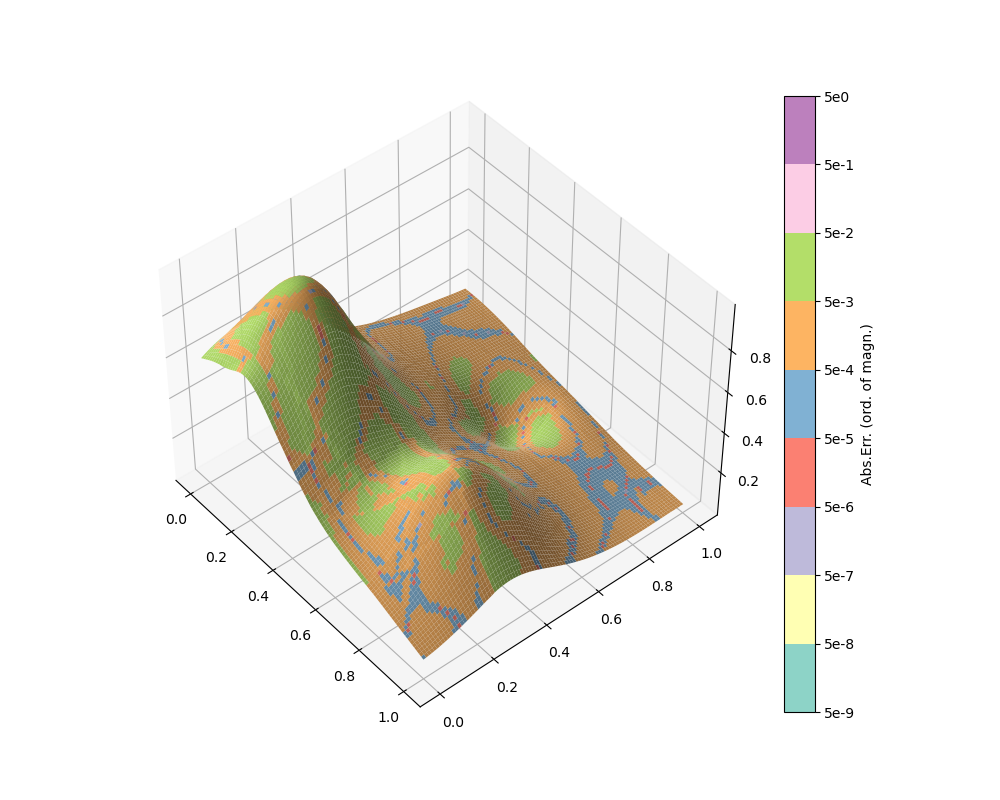}
     }
     \subcaptionbox{VSKs-$f$}{
    \includegraphics[trim={2.3cm 1.6cm 2.45cm 1.9cm},clip,width=0.325\textwidth]{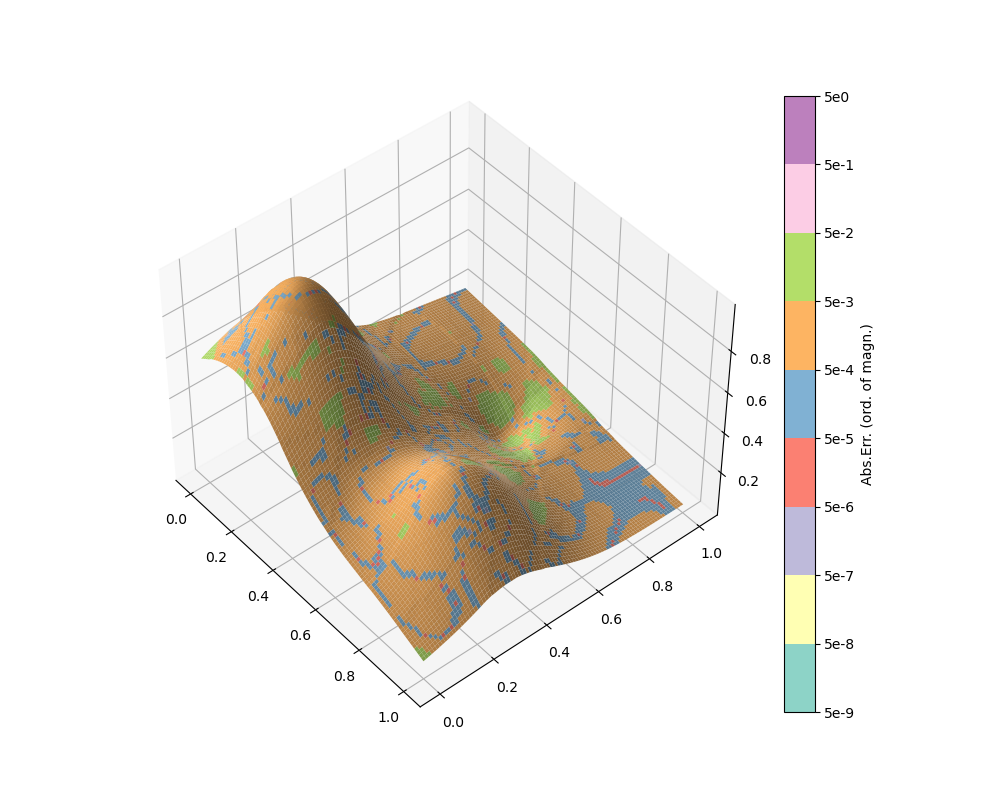}
    }
    \caption{{Interpolation results for $f_1$ (see Figure \ref{fig:f1_topview_comparison}), colored w.r.t. interpolation error (absolute error). In the bottom row, colors describe the order of magnitude of the absolute error. Shared colorscales for subfigures ($b$)-($d$) and for subfigures ($e$)-($g$).}}
    \label{fig:f1_surface_errors_horizontal}
 \end{figure}
 
  \begin{figure}[htbp!]
     \centering
     \subcaptionbox{Target}{
     \includegraphics[trim={2.3cm 1.6cm 2.45cm 1.9cm},clip,width=0.38\textwidth]{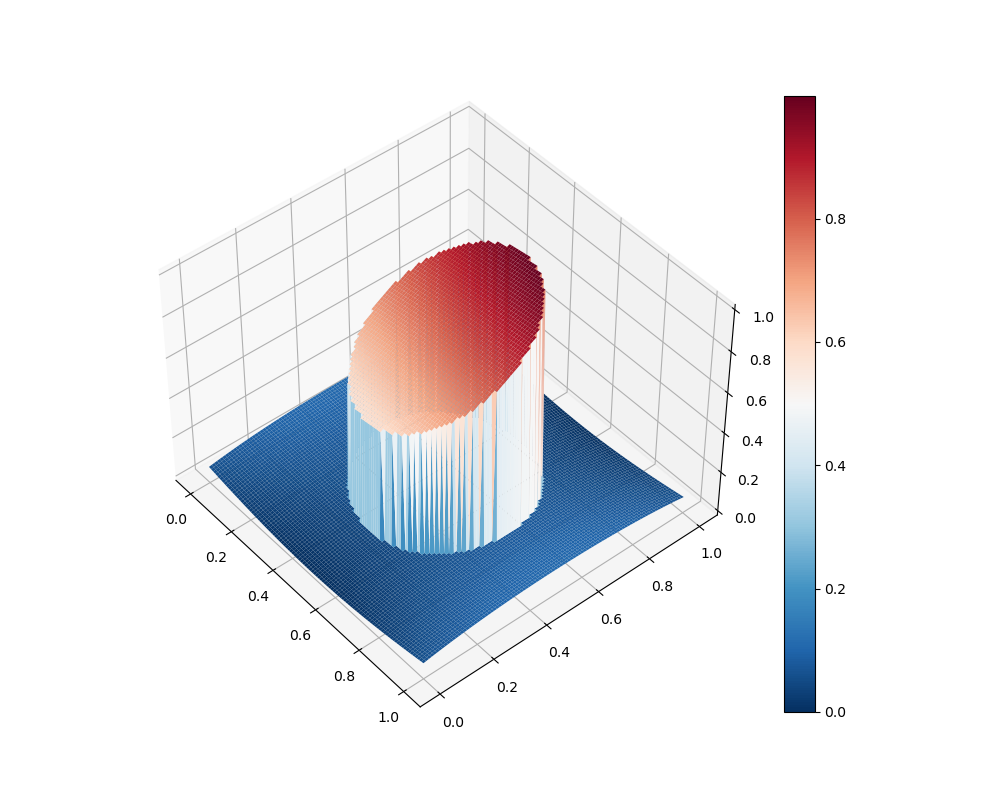}
     }
     \\
     \subcaptionbox{FSKs}{
     \includegraphics[trim={2.3cm 1.6cm 5.55cm 1.9cm},clip,width=0.28\textwidth]{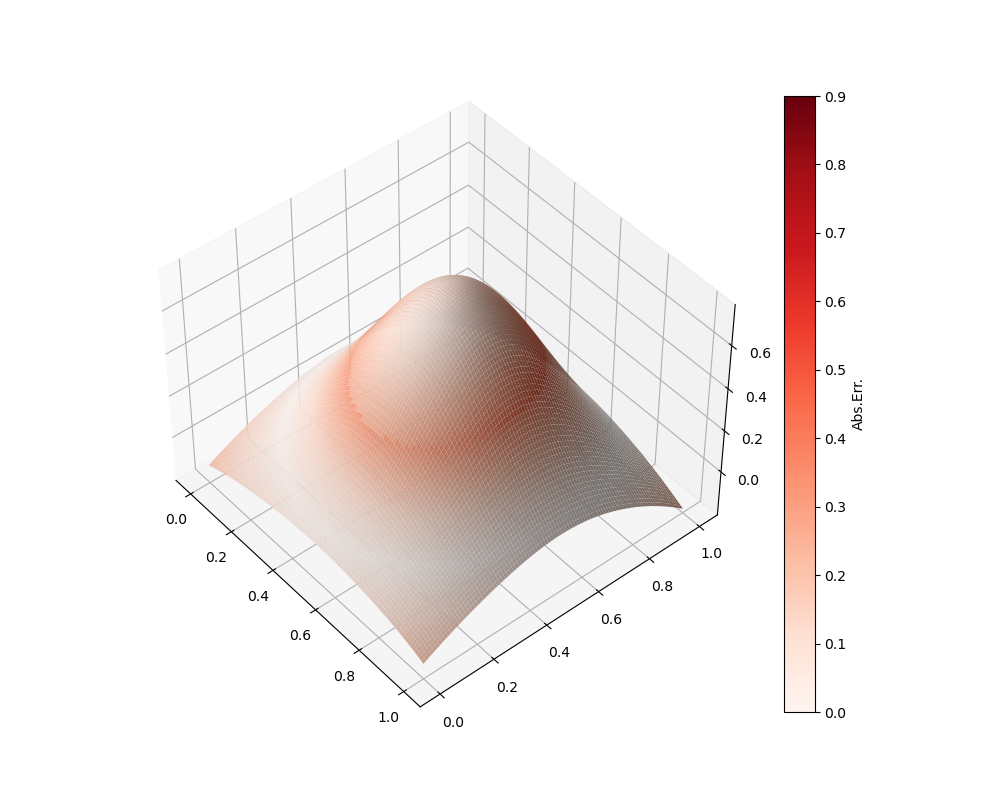}
     }
     \subcaptionbox{$\delta$NN-VSKs}{
     \includegraphics[trim={2.3cm 1.6cm 5.55cm 1.9cm},clip,width=0.28\textwidth]{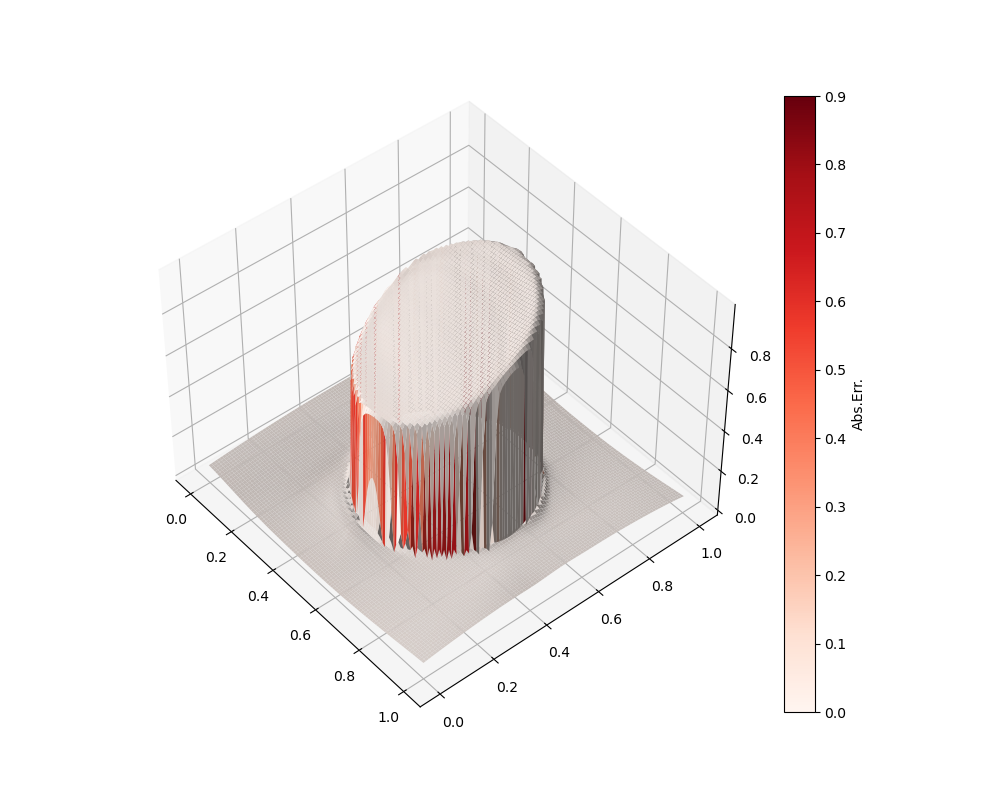}
     }
     \subcaptionbox{VSKs-$f$}{
     \includegraphics[trim={2.3cm 1.6cm 2.45cm 1.9cm},clip,width=0.325\textwidth]{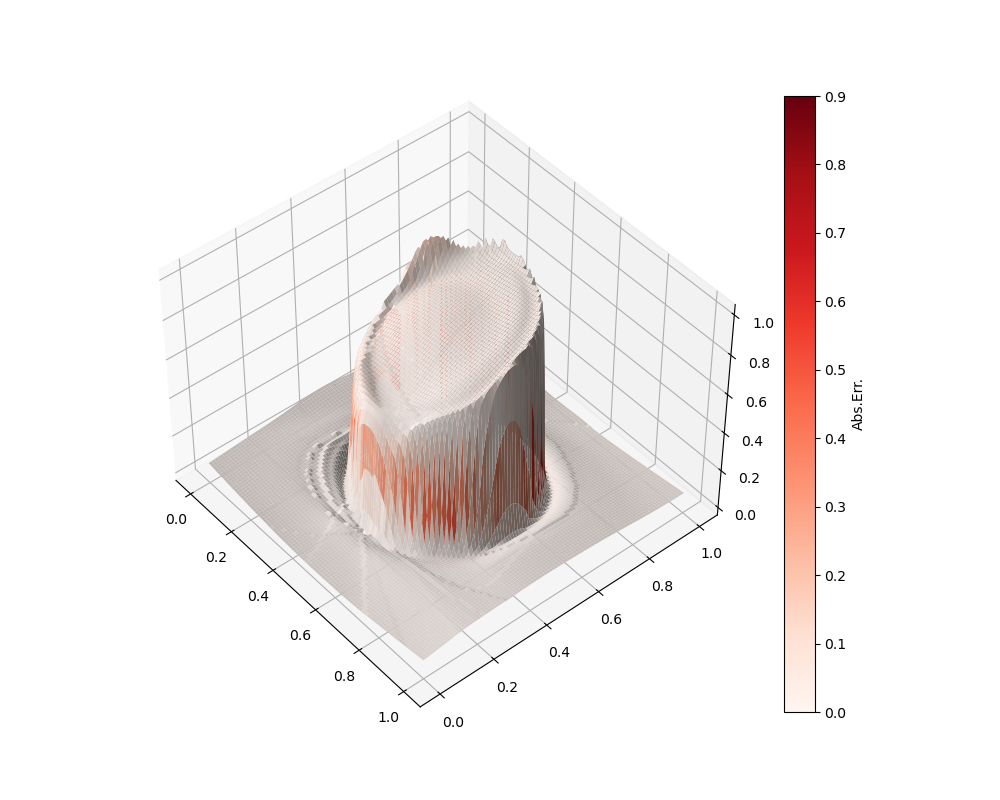}
     }
     \\
     \subcaptionbox{FSKs}{
     \includegraphics[trim={2.3cm 1.6cm 5.55cm 1.9cm},clip,width=0.28\textwidth]{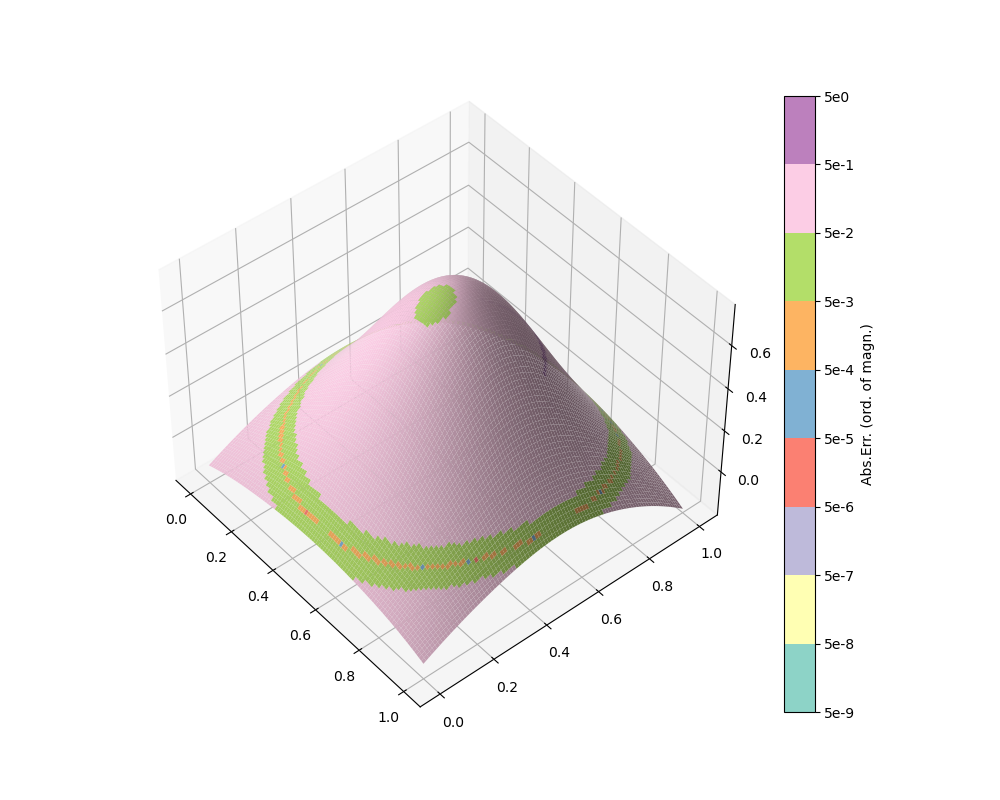}
     }
     \subcaptionbox{$\delta$NN-VSKs}{
     \includegraphics[trim={2.3cm 1.6cm 5.55cm 1.9cm},clip,width=0.28\textwidth]{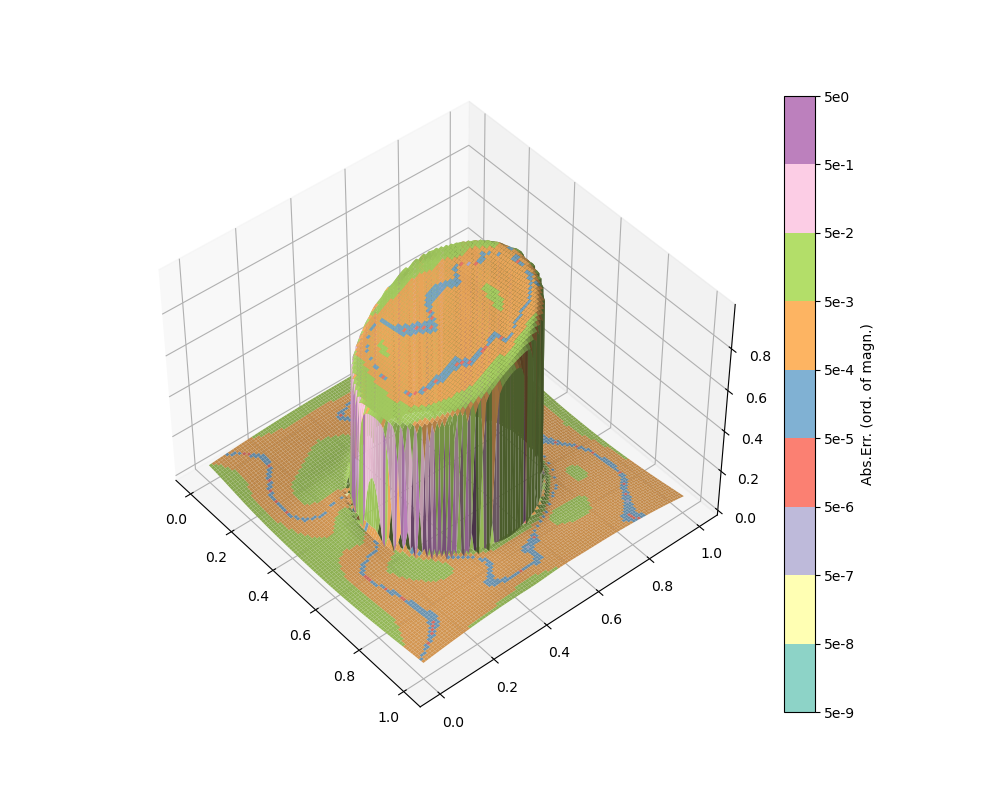}
     }
     \subcaptionbox{VSKs-$f$}{
     \includegraphics[trim={2.3cm 1.6cm 2.45cm 1.9cm},clip,width=0.325\textwidth]{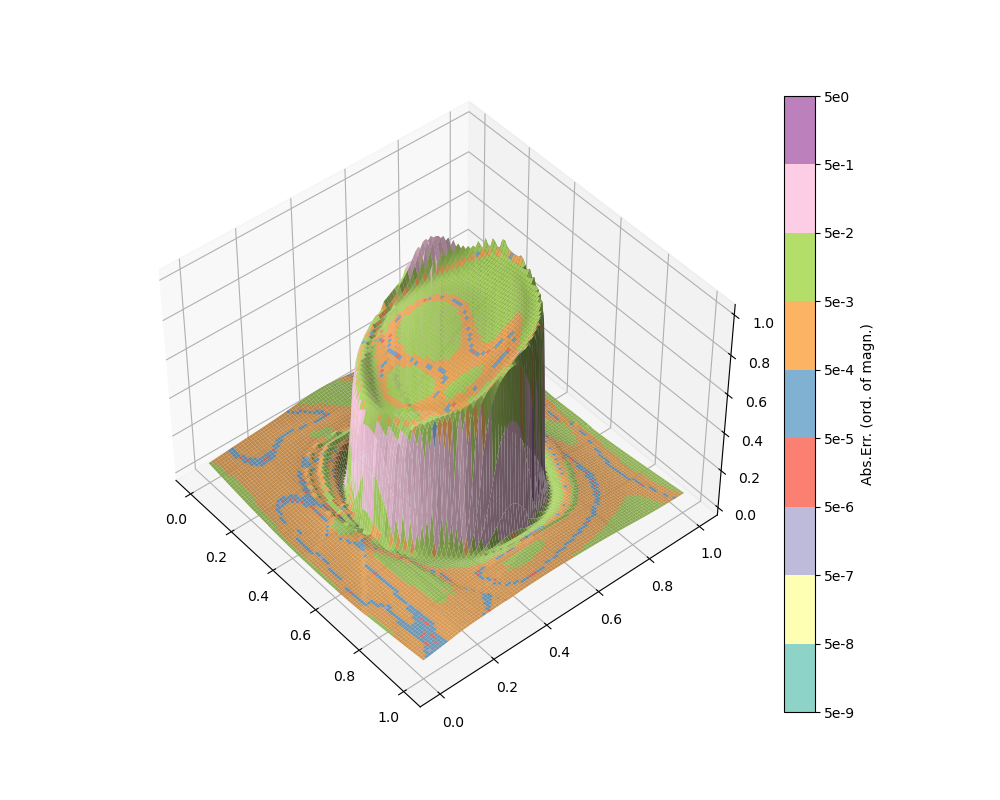}
     }
     \caption{{Interpolation results for $f_2$ (see Figure \ref{fig:f2_topview_comparison}), colored w.r.t. interpolation error (absolute error). In the bottom row, colors describe the order of magnitude of the absolute error. Shared color bars for subfigures ($b$)-($d$) and for subfigures ($e$)-($g$).}}
     \label{fig:f2_surface_errors_horizontal}
 \end{figure}

\begin{figure}[htbp!]
     \centering
     \subcaptionbox{Target}{
     \includegraphics[trim={2.3cm 1.6cm 2.45cm 1.9cm},clip,width=0.38\textwidth]{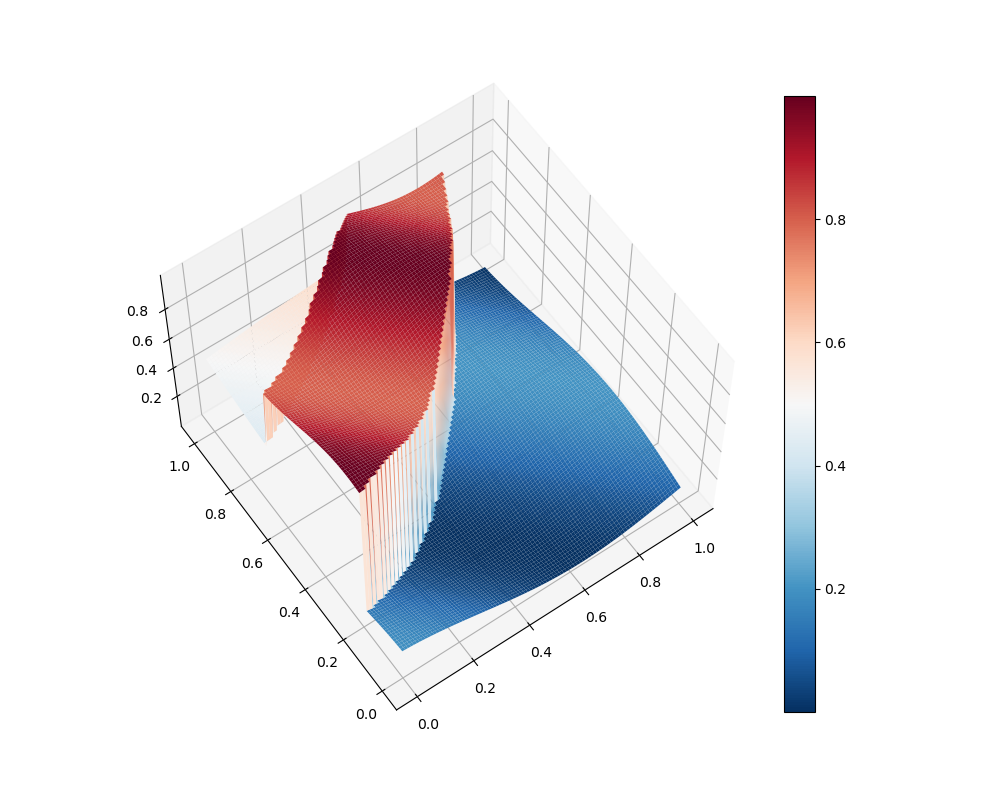}
     }
     \\
     \subcaptionbox{FSKs}{
     \includegraphics[trim={2.3cm 1.6cm 5.55cm 1.9cm},clip,width=0.28\textwidth]{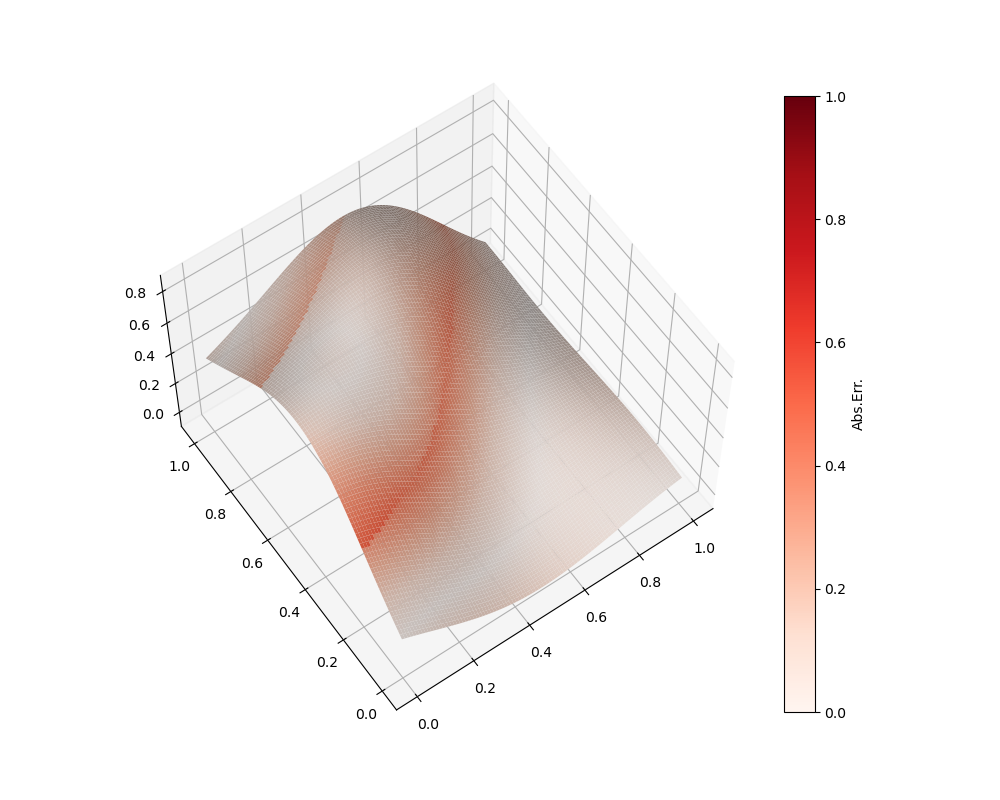}
     }
     \subcaptionbox{$\delta$NN-VSKs}{
     \includegraphics[trim={2.3cm 1.6cm 5.55cm 1.9cm},clip,width=0.28\textwidth]{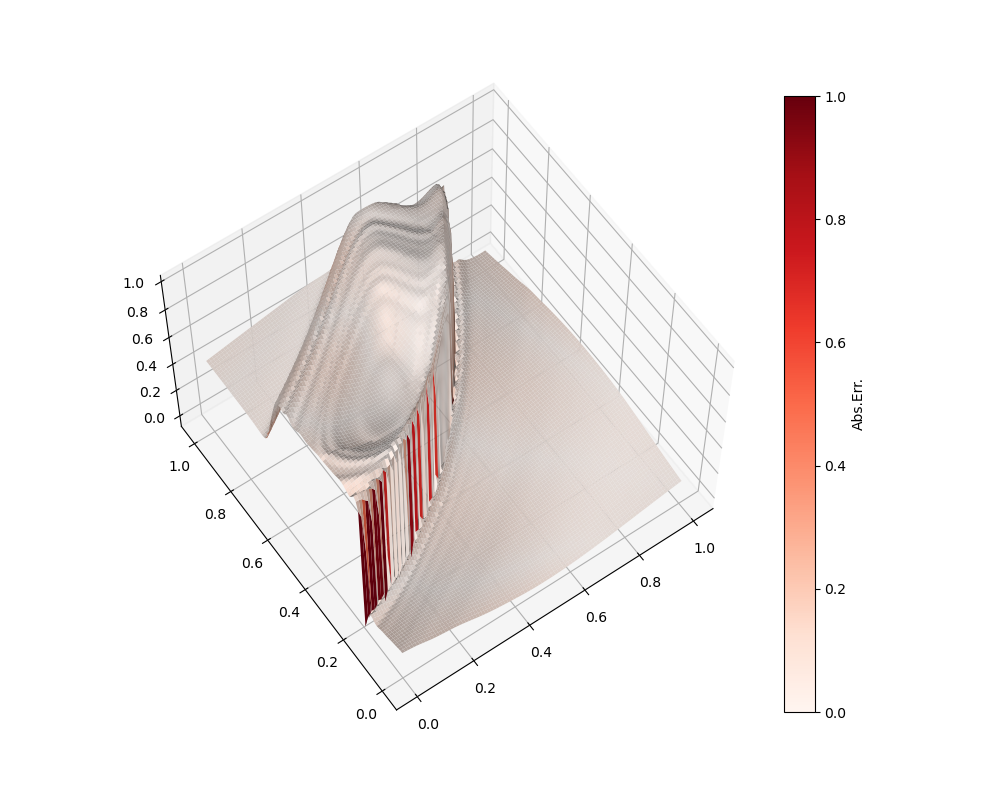}
     }
     \subcaptionbox{VSKs-$f$}{
     \includegraphics[trim={2.3cm 1.6cm 2.45cm 1.9cm},clip,width=0.325\textwidth]{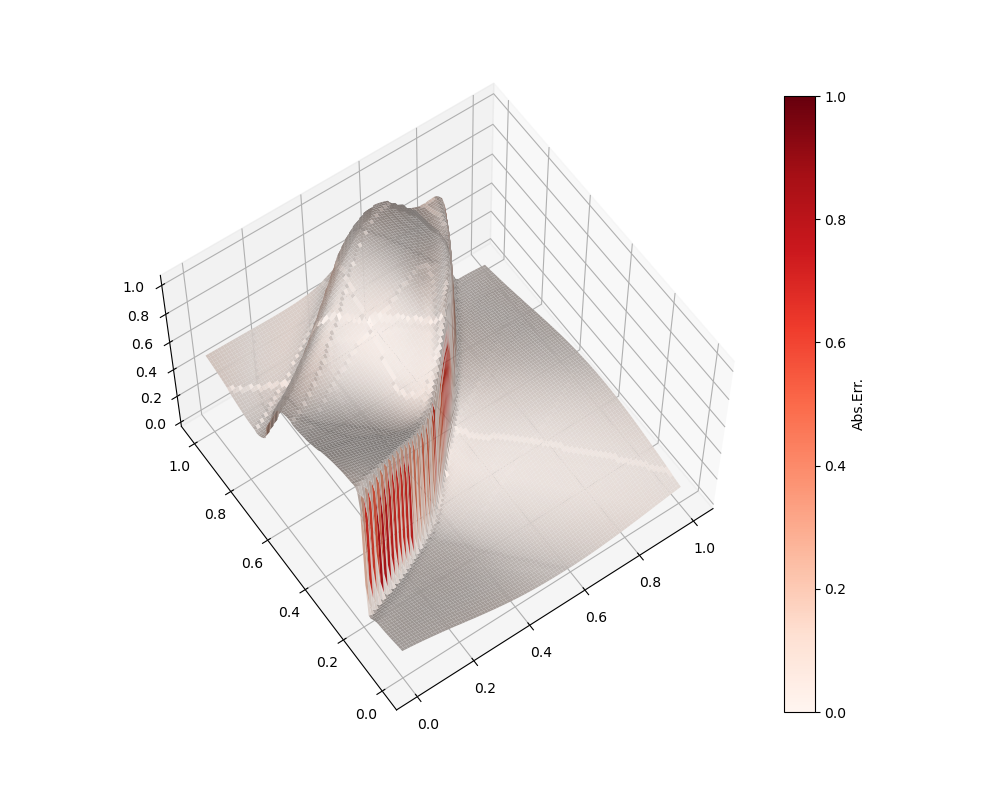}
     }
     \\
     \subcaptionbox{FSKs}{
     \includegraphics[trim={2.3cm 1.6cm 5.55cm 1.9cm},clip,width=0.28\textwidth]{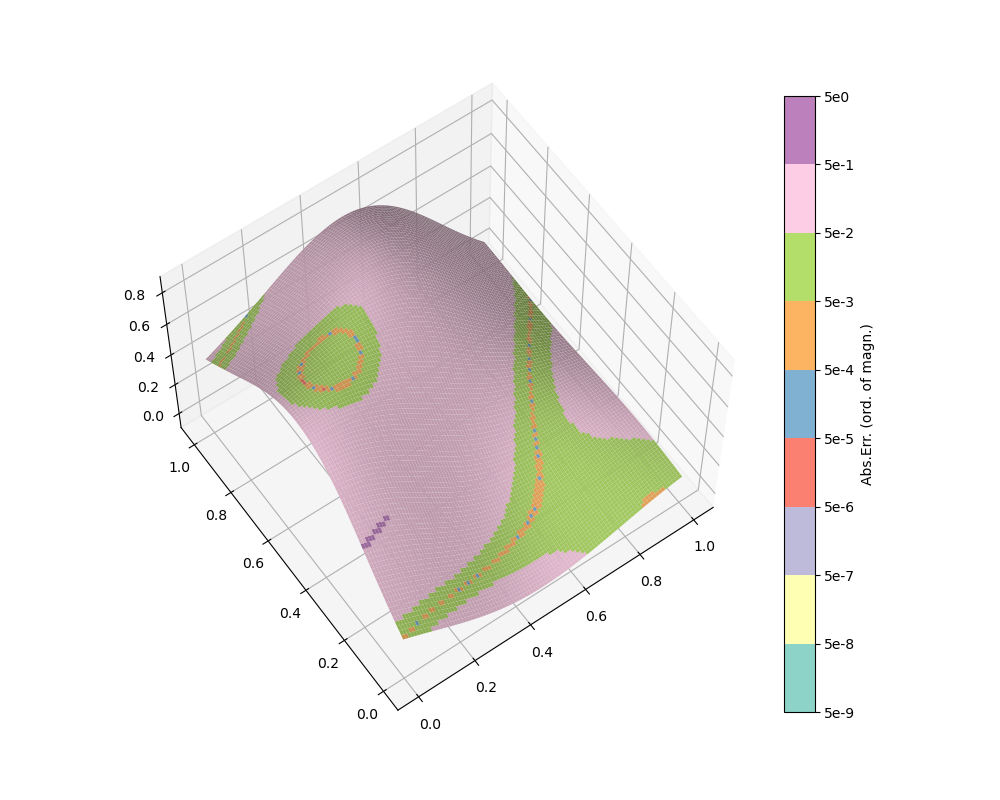}
     }
     \subcaptionbox{$\delta$NN-VSKs}{
     \includegraphics[trim={2.3cm 1.6cm 5.55cm 1.9cm},clip,width=0.28\textwidth]{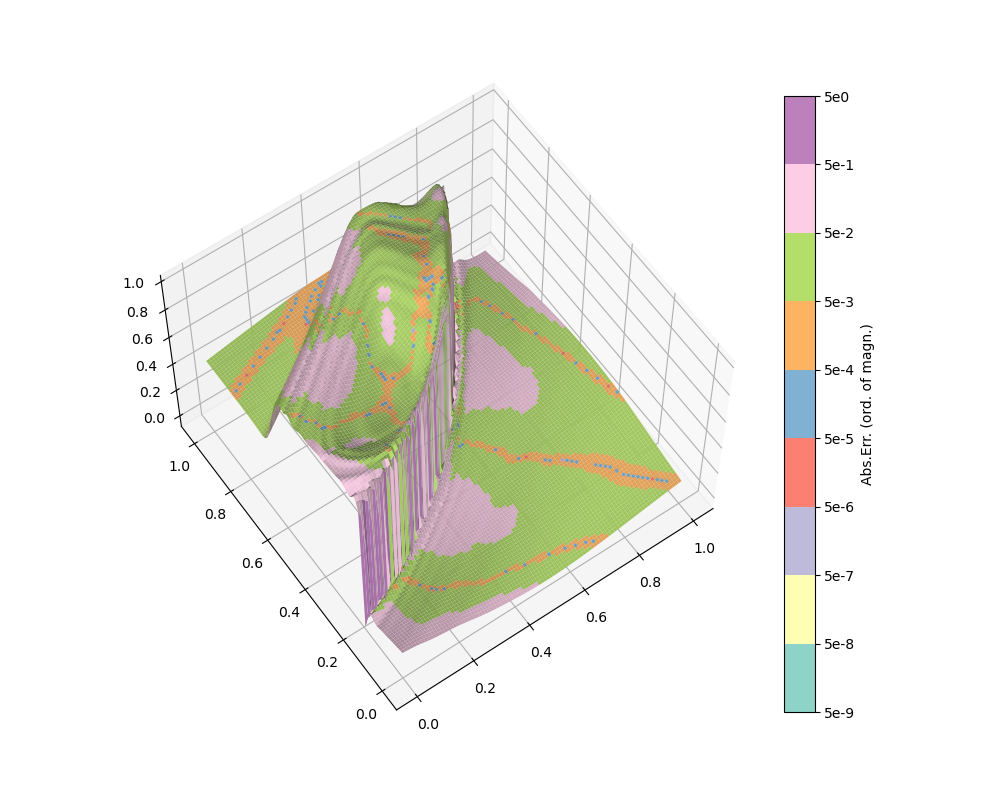}
     }
     \subcaptionbox{VSKs-$f$}{
     \includegraphics[trim={2.3cm 1.6cm 2.45cm 1.9cm},clip,width=0.325\textwidth]{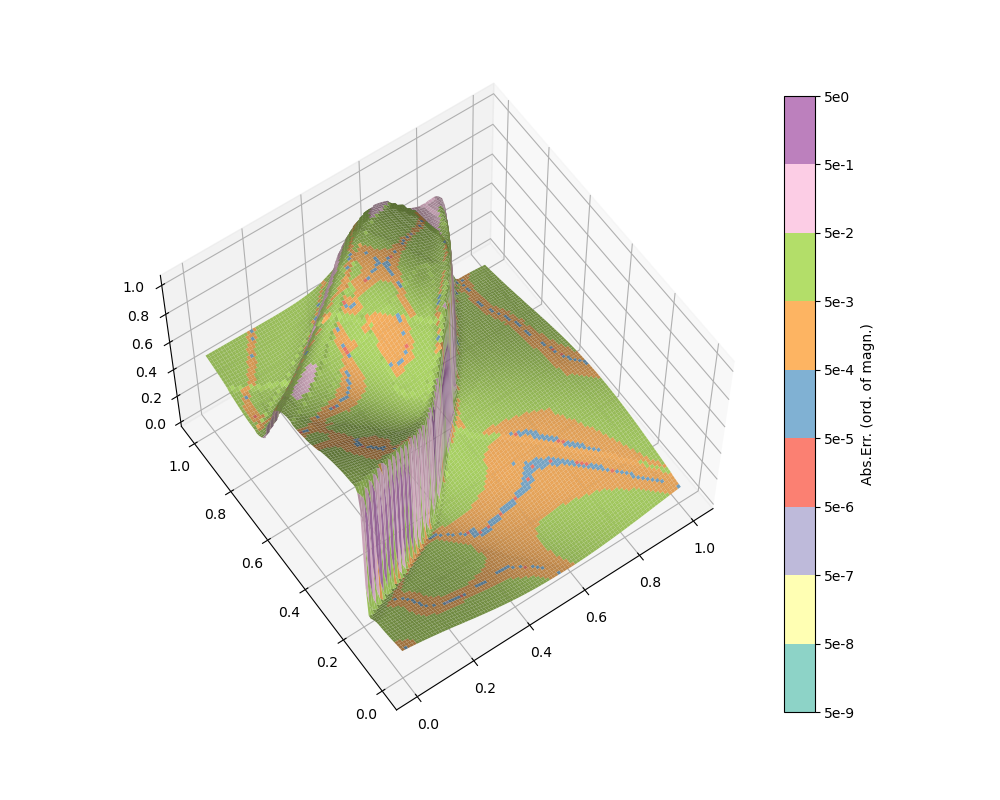}
     }
     \caption{{Interpolation results for $f_3$ (see Figure \ref{fig:f3_topview_comparison}), colored w.r.t. interpolation error (absolute error). In the bottom row, colors describe the order of magnitude of the absolute error. Shared color bars for subfigures ($b$)-($d$) and for subfigures ($e$)-($g$).}}
     \label{fig:f3_surface_errors_horizontal}
 \end{figure}

{Concluding the analyses of the interpolation results for the three test functions $f_1$, $f_2$, $f_3$, we measure the NN models training time, for both the $\delta$NN-VSK method and the VSK-$f$ method. All the training procedures have been performed on a Laptop LENOVO 21HFS04Q00 ThinkPad P14s Gen 4, 12-core CPU (4-mt/8-st) model 13th Gen Intel Core i7-1360P (bits: 64), 32GB RAM, NVIDIA RTX A500 Laptop GPU. The NN model training time is shown in Table \ref{tab:tabella_1_times} where we report the training time as the average over 10 runs together with their standard deviation.

We measure the time spent for training the NN models because it is the only extra computational cost with respect to FSKs. Moreover, we recall that the $\delta$NN-VSK method is trained in a deterministic way, on the whole interpolation data; then, all the available training epochs are used, because no other stopping criteria are considered. On the other hand, the training of the $\delta$NN as scaling function in the VSK-$f$ methods follows the classic NN training rules (see \ref{sec:deltaNN_arch} for details); it is then a stochastic training, and early stopping regularization can activate before the maximum number of epochs is reached (see the cases $f_2$, $n=1521$, and $f_3$, $n=1089$ in Table \ref{tab:tabella_1_times}).

In general, we observe that these training procedures are relatively fast: at most 2 minutes and a half for $\delta$NN-VSKs and at most $5$ minutes for VSKs-$f$. Therefore, the best results obtained using the proposed methods, especially for discontinuous functions, are worth the extra computational cost in terms of time. Concerning the scalability with respect to $n$, the $\delta$NN-VSK method performs better, because it is deterministic and not based on mini-batches (mini-batches may increase the number of inner iteration for epochs in NN training procedures).
}

\begin{table}[htb!]
\centering
\resizebox{1.\textwidth}{!}{
\begin{tabular}{|c|c|cc|cc|}
\hline 
& & \multicolumn{2}{|c|}{Training Time (min:sec)} & \multicolumn{2}{|c|}{Training Epochs (done/max)} \\
\hline 
& $n$ & $\delta$NN-VSKs & VSKs-$f$ & $\delta$NN-VSKs & VSKs-$f$ \\
\hline 
\multirow{3}{*}{$f_1$} 
& 729 & 01:04.28 $\pm$ 0.30s & 02:41.68 $\pm$ 0.79s & 2000/2000 & 1000/1000 \\
& 1089 & 01:36.96 $\pm$ 0.22s & 03:47.06 $\pm$ 0.98s & 2000/2000 & 1000/1000 \\
& 1521 & 02:19.92 $\pm$ 0.26s & 05:06.02 $\pm$ 0.99s & 2000/2000 & 1000/1000 \\
\hline
\multirow{3}{*}{$f_2$} 
& 729 & 01:09.13 $\pm$ 1.17s & 02:41.77 $\pm$ 0.60s & 2000/2000 & 1000/1000 \\
& 1089 & 01:44.41 $\pm$ 0.32s & 03:46.98 $\pm$ 0.48s & 2000/2000 & 1000/1000 \\
& 1521 & 02:33.06 $\pm$ 0.28s & 04:44.27 $\pm$ 0.63 & 2000/2000 & 928/1000 \\
\hline
\multirow{3}{*}{$f_3$} 
& 729 & 01:06.84 $\pm$ 0.12s & 02:42.33 $\pm$ 0.47s & 2000/2000 & 1000/1000 \\
& 1089 & 01:43.88 $\pm$ 0.25s & 03:34.77 $\pm$ 0.99s & 2000/2000 & 947/1000 \\
& 1521 & 02:32.31 $\pm$ 0.28s & 05:05.35 $\pm$ 0.87s & 2000/2000 & 1000/1000 \\
\hline 
\end{tabular}
}
\caption{{Training times in minutes and seconds for the methods $\delta$NN-VSKs and VSKs-$f$ on the test functions $f_1$, $f_2$, and $f_3$. Training times were computed by running the methods 10 times and computing the average and the standard deviation.}
}
\label{tab:tabella_1_times}
\end{table}

\subsection{{Real-world test case}}\label{sec:acetone}

{
In this subsection, we apply the proposed $\delta$NN-based VSK methods to approximate a discontinuous function obtained from real data. Specifically, we consider the phase transition phenomenon of acetone. For this medium (like other ones), the state of matter is characterized by some parameters, such as the density $\rho$. Nonetheless, from a physical point of view, the density can be described as a function of the temperature $T$ and the pressure $p$; then, we can define the function $f_4$:
\begin{equation*}
	\rho=f_4(T, p)\,.
\end{equation*}

We are interested in approximating $f_4$ because, for a wide range of media, it has been observed via experimental measurements that the density often presents discontinuity interfaces that separate two or more regions of the $(T,p)$ plane, representing different states of matter. The points belonging to these discontinuity interfaces are called equilibrium points for the states of matter they separate. However, phase transitions can happen also continuously, crossing a region of so-called supercritical points of the $(T,p)$ plane. This classification of phase transition phenomena is called Ehrenfest classification \cite{Jaeger1998_PHASETRANSITION}. See Figures \ref{fig:f4_topview_comparison}-(a) and \ref{fig:f4_surface_errors_horizontal}-(a) for a reconstruction of the density function $f_4$ of acetone ($T$ and $p$ are normalized); see Figure \ref{fig:acetone} for an example of experimental measures and a description of the state of matter. Looking at these figures, we observe that the function is characterized by a discontinuity interface that is a sort of nonlinear ``rip''. 
}

\begin{figure}[htb]
    \centering
    \includegraphics[width=0.65\textwidth]{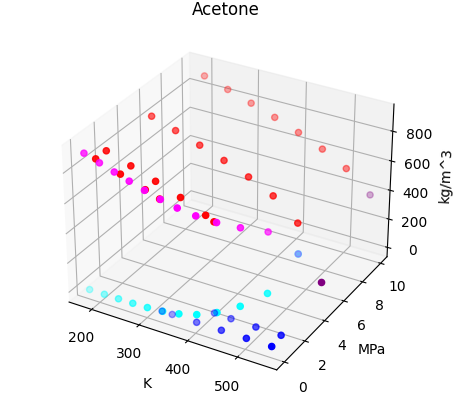}
    \caption{{$(T,p,\rho)$ points of acetone taken from \cite{engtoolbox_PHASETRANSITION_ACETONE}; red dots correspond to liquid state, magenta to liquid state at equilibrium, blue to gas state, cyan to gas state at equilibrium, and purple to supercritical points.}
    }
    \label{fig:acetone}
\end{figure}

{
The discontinuous behavior of $f_4$ makes it an interesting new test function, useful for extending the analysis to a real-world scenario. Then, we approximate the acetone's density function by using FSK, $\delta$NN-VSK, and VSK-$f$ methods, with respect to the same values $n$ of interpolation points used in Section \ref{sintetici}. 

Also for this case, we observe that the $\delta$NN-VSK and VSK-$f$ methods improves on FSKs in the approximation accuracy (see Table \ref{tab:tabella_2_acetone} and Figures \ref{fig:f4_topview_comparison} and \ref{fig:f4_surface_errors_horizontal}). In particular, even if the approximation errors of FSKs are not particularly bad, the VSKs-based methods show very good abilities in correctly approximating the nonlinear rip characterizing the discontinuity. This property is very important for the real-world scenario we are considering.

Concerning the scaling function learned via $\delta$NN, we observe similar behaviors to the ones observed in Section \ref{sintetici}. The $\delta$NN-VSK method learns a nearly-optimal scaling functions that resemble the target in the form of $\bar{f}\approx \gamma f$, with $\gamma <0$. On the other hand, the VSK-$f$ method continues to support the theoretical claims of the safety in having $\bar{f}\approx f$; refer to the bottom rows of Figure \ref{fig:f4_topview_comparison}.
}

\begin{table}[htb]
\centering
\resizebox{1.\textwidth}{!}{
\begin{tabular}{|c|ccc|ccc|ccc|ccc|}
\hline 
& \multicolumn{3}{c|}{Interpolation} & \multicolumn{3}{c|}{MAE} & \multicolumn{3}{c|}{MSE} & \multicolumn{3}{c|}{SSIM} \\
\hline 
& $n$ & Kernel & $\varepsilon$ & FSKs & $\delta$NN-VSKs & VSKs-$f$ & FSKs & $\delta$NN-VSKs & FSKs & FSKs & $\delta$NN-VSKs & VSKs-$f$  \\
\hline
\multirow{3}{*}{$f_4$} 
& 729 & Mat.$C^2$ & 0.06 & 7.16\rm{e}-2 & 1.03\rm{e}-2 & 9.23\rm{e}-3 & 1.38\rm{e}-2 & 3.34\rm{e}-3 & 3.26\rm{e}-3 &
0.9000 & 0.9777 & 0.9853 \\
& 1089 & Mat.$C^2$ & 0.12 & 4.94\rm{e}-2 & 7.98\rm{e}-3 & 8.29\rm{e}-3 & 1.05\rm{e}-2 & 3.45\rm{e}-3 & 3.35\rm{e}-3 &
0.9158 & 0.9852 & 0.9902 \\
& 1521 & Mat.$C^2$ & 0.48 & 2.54\rm{e}-2 & 8.10\rm{e}-3 & 7.22\rm{e}-3 & 6.06\rm{e}-3& 3.43\rm{e}-3 & 2.84\rm{e}-3 &
0.9251 & 0.9856 & 0.9883 \\
\hline 
\end{tabular}
}
\caption{{MAEs, MSEs, and SSIMs for the test function $f_4$. }
}
\label{tab:tabella_2_acetone}
\end{table}

\begin{figure}[htbp!]
    \centering
    \subcaptionbox{Target}{
    \includegraphics[trim={2.3cm 1.6cm 2.45cm 1.9cm},clip,width=0.38\textwidth]{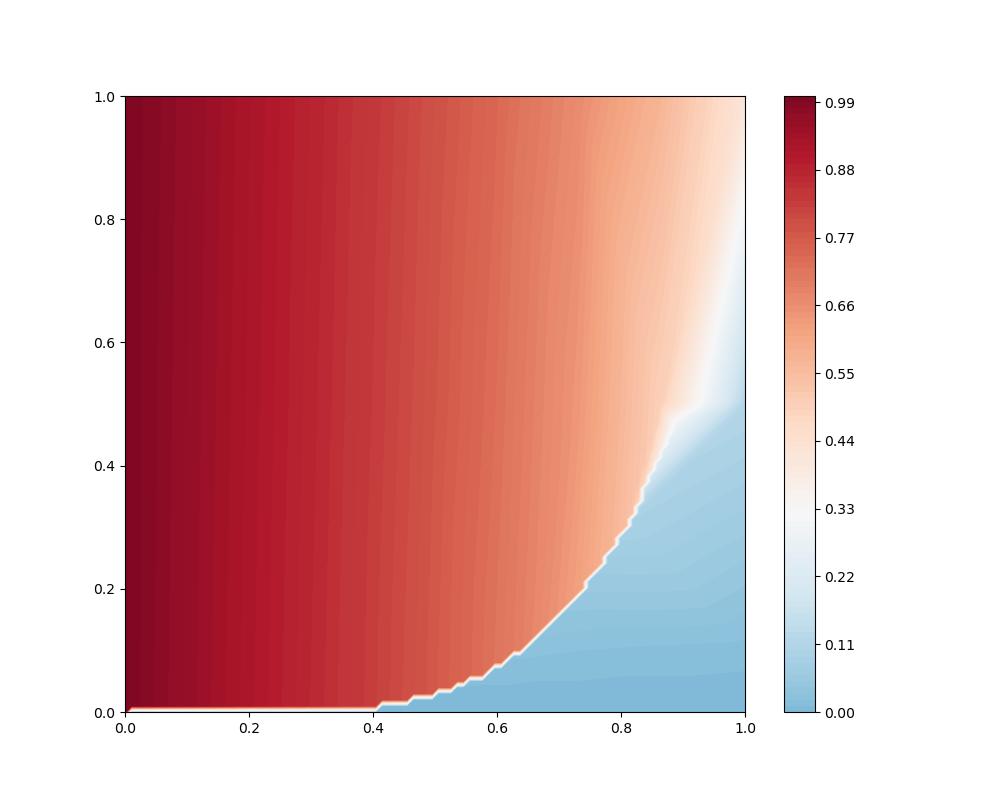}
    }
    \subcaptionbox{FSKs}{
    \includegraphics[trim={2.3cm 1.6cm 2.45cm 1.9cm},clip,width=0.38\textwidth]{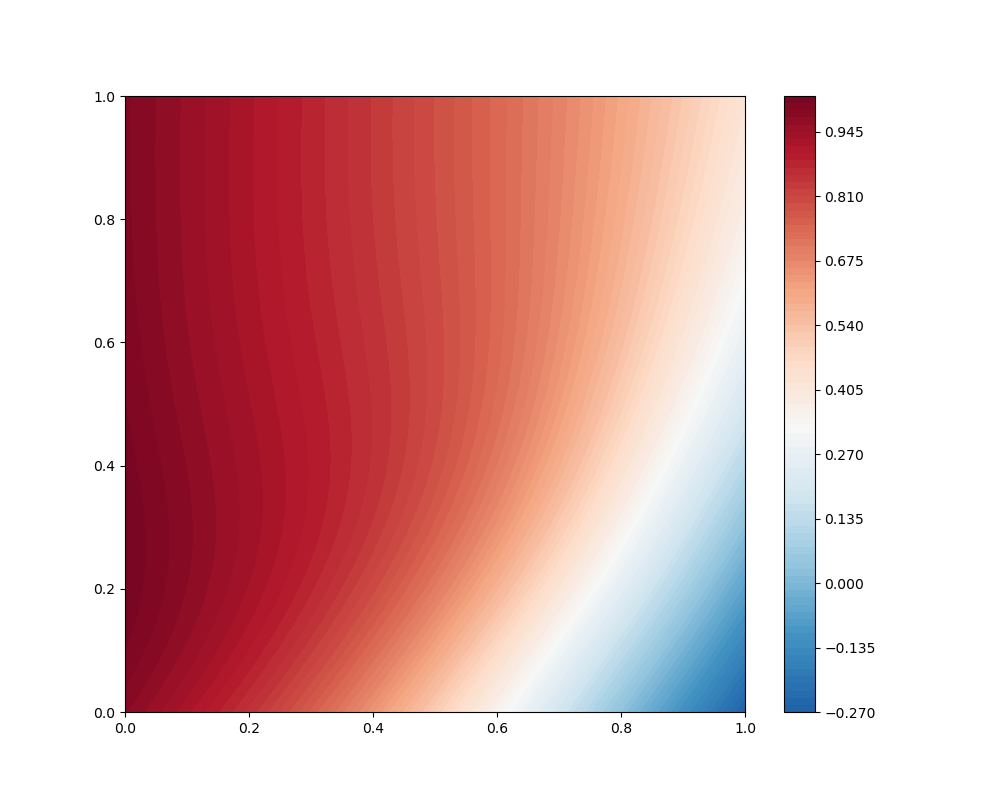}
    }
    \\
    \subcaptionbox{$\delta$NN-VSKs}{
    \includegraphics[trim={2.3cm 1.6cm 2.45cm 1.9cm},clip,width=0.38\textwidth]{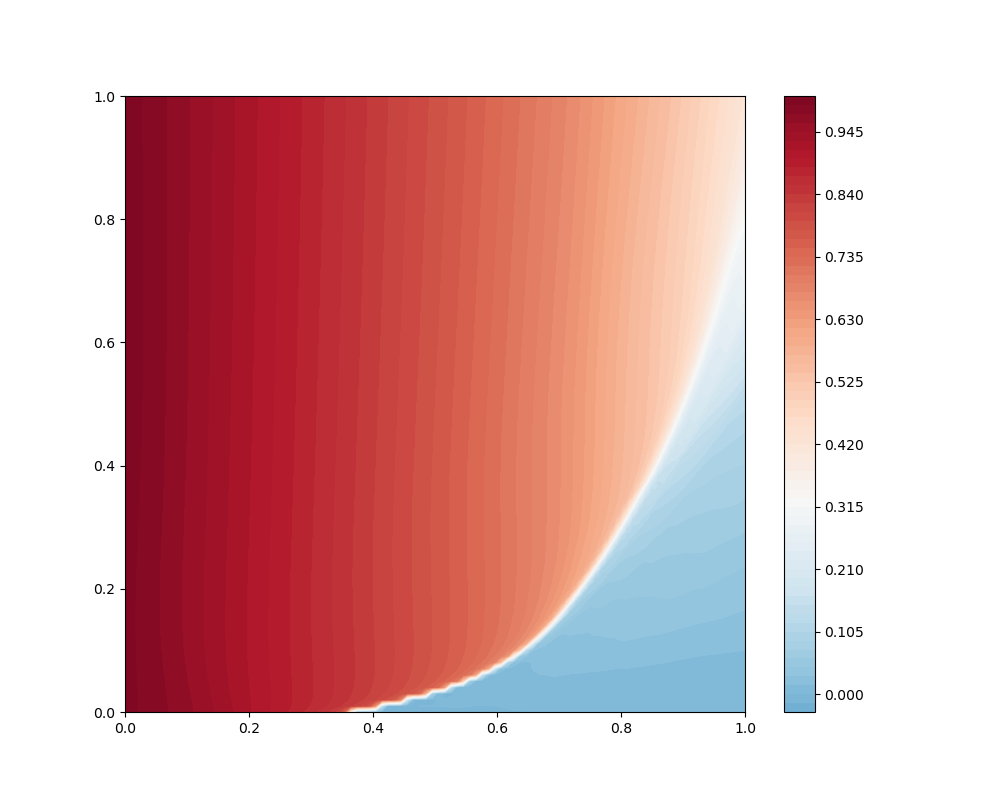}
    }
    \subcaptionbox{VSKs-$f$}{
    \includegraphics[trim={2.3cm 1.6cm 2.45cm 1.9cm},clip,width=0.38\textwidth]{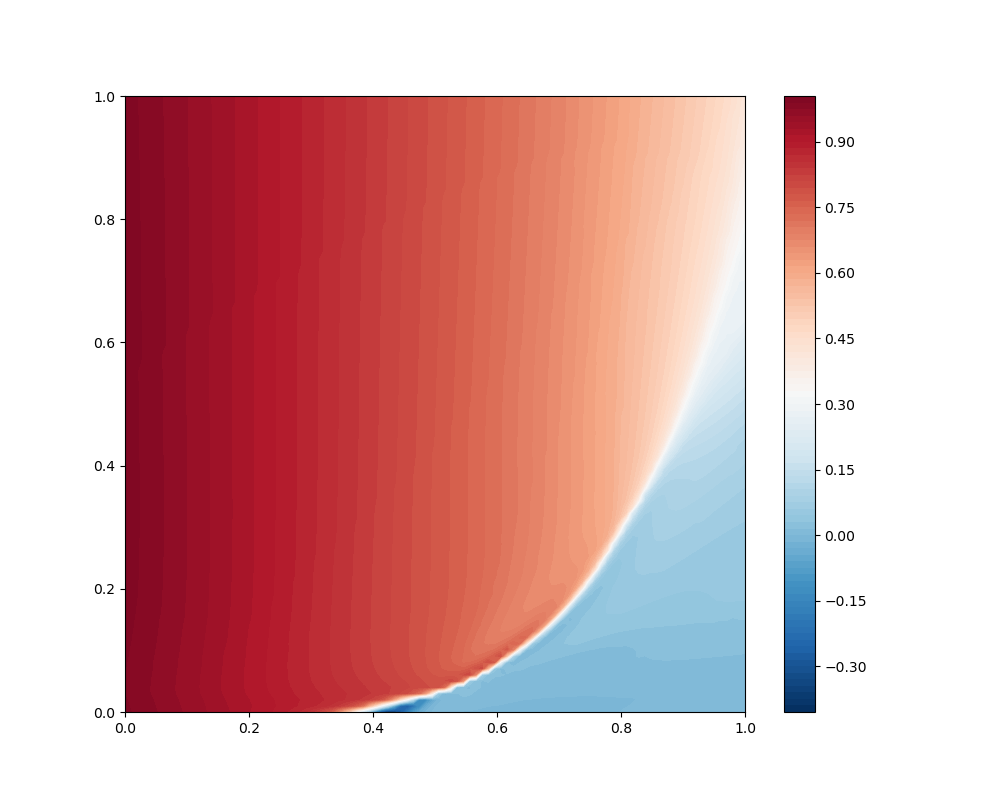}
    }
    \\
    \subcaptionbox{$\delta$NN-VSKs,\\ scaling func.}{
    \includegraphics[trim={2.3cm 1.6cm 2.45cm 1.9cm},clip,width=0.38\textwidth]{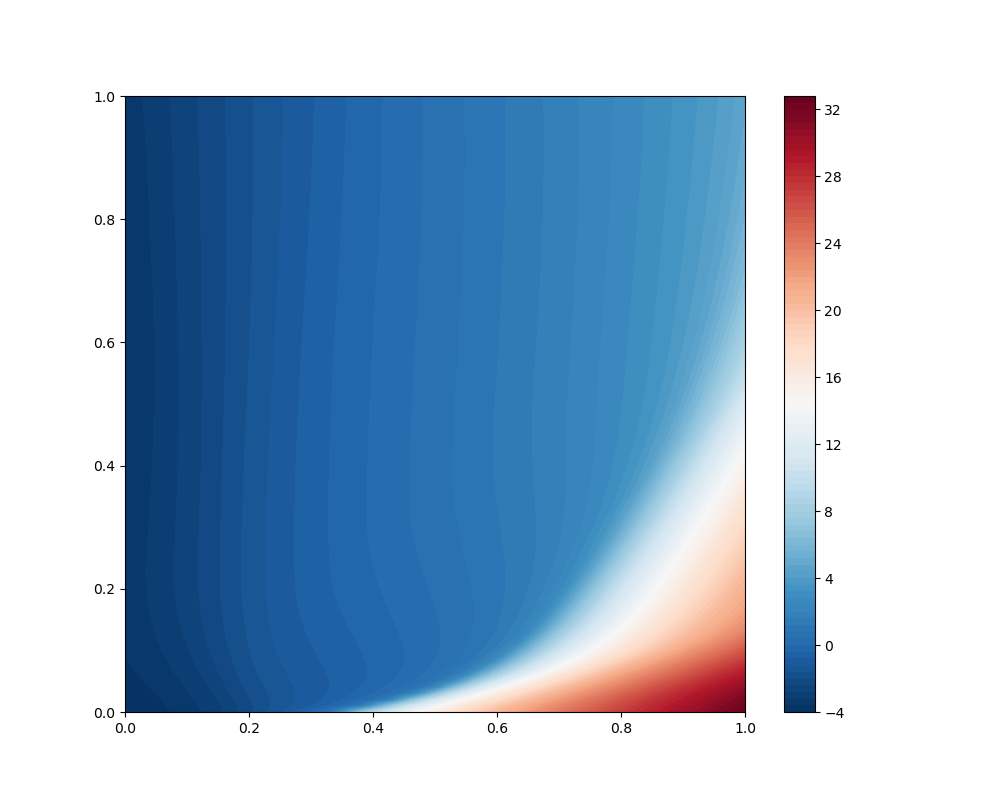}
    }
    \subcaptionbox{VSKs-$f$,\\ scaling func.}{
    \includegraphics[trim={2.3cm 1.6cm 2.45cm 1.9cm},clip,width=0.38\textwidth]{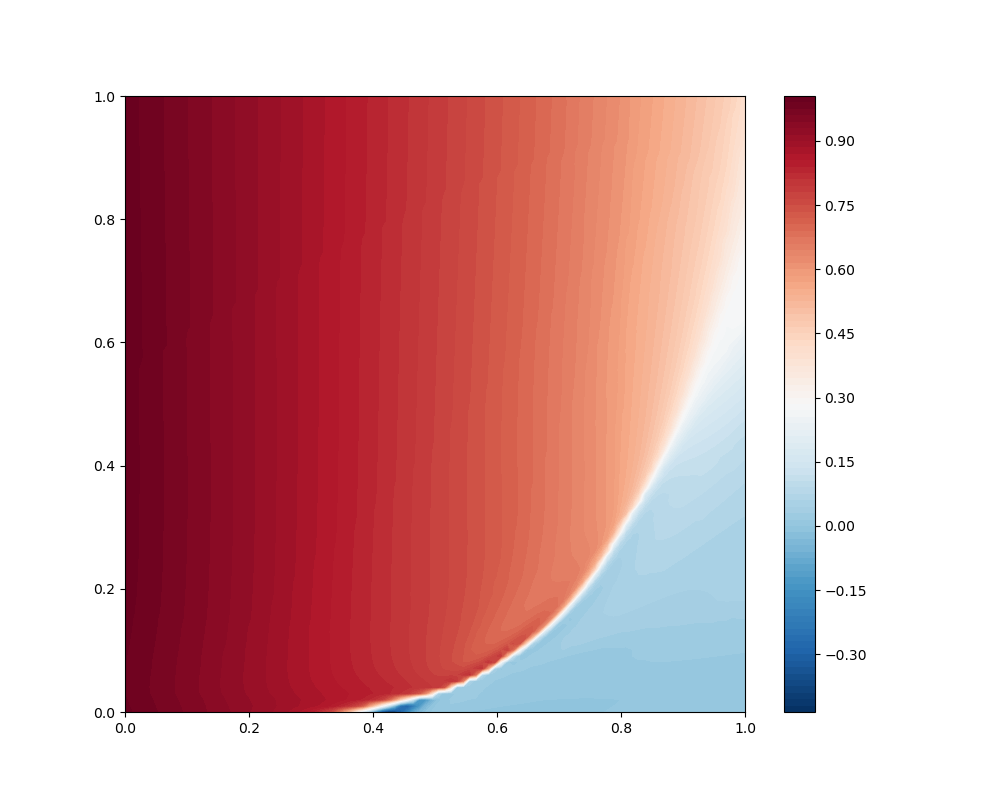}
    }
    \caption{{Interpolation results for $f_4$. Case $n=33^2=1089$ interpolation points, Mat\'ern-$C^2$ Kernel. Shared color bars for subfigures ($a$)-($d$); custom color bar for the scaling functions (subfigures ($e$) and ($f$)).}}
    \label{fig:f4_topview_comparison}
\end{figure}

 \begin{figure}[htbp!]
     \centering
     \subcaptionbox{Target}{
     \includegraphics[trim={2.3cm 1.6cm 2.45cm 1.9cm},clip,width=0.38\textwidth]{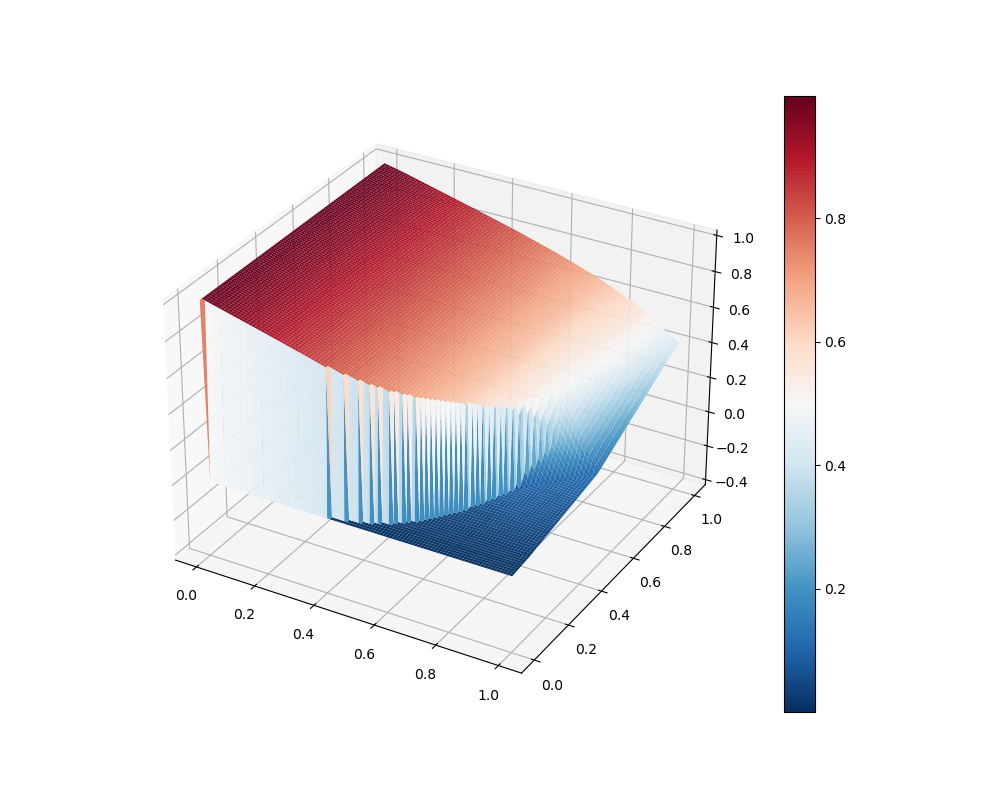}
     }
     \\
          \subcaptionbox{FSKs}{
     \includegraphics[trim={2.3cm 1.6cm 5.55cm 1.9cm},clip,width=0.28\textwidth]{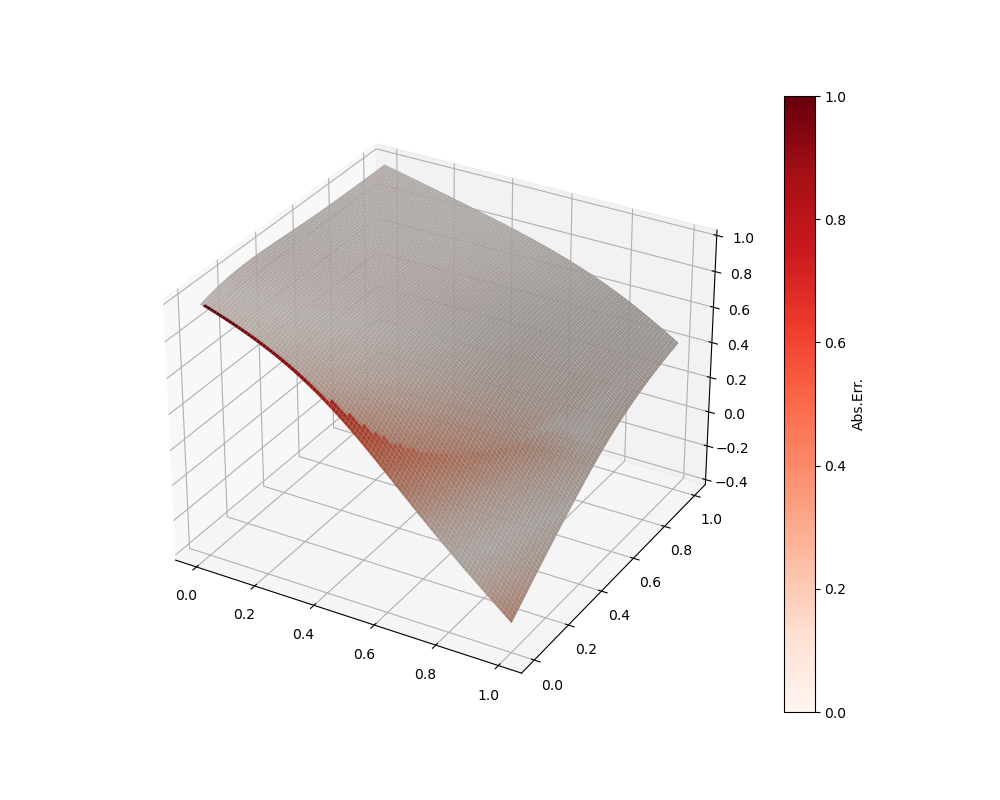}
     }
     \subcaptionbox{$\delta$NN-VSKs}{
     \includegraphics[trim={2.3cm 1.6cm 5.55cm 1.9cm},clip,width=0.28\textwidth]{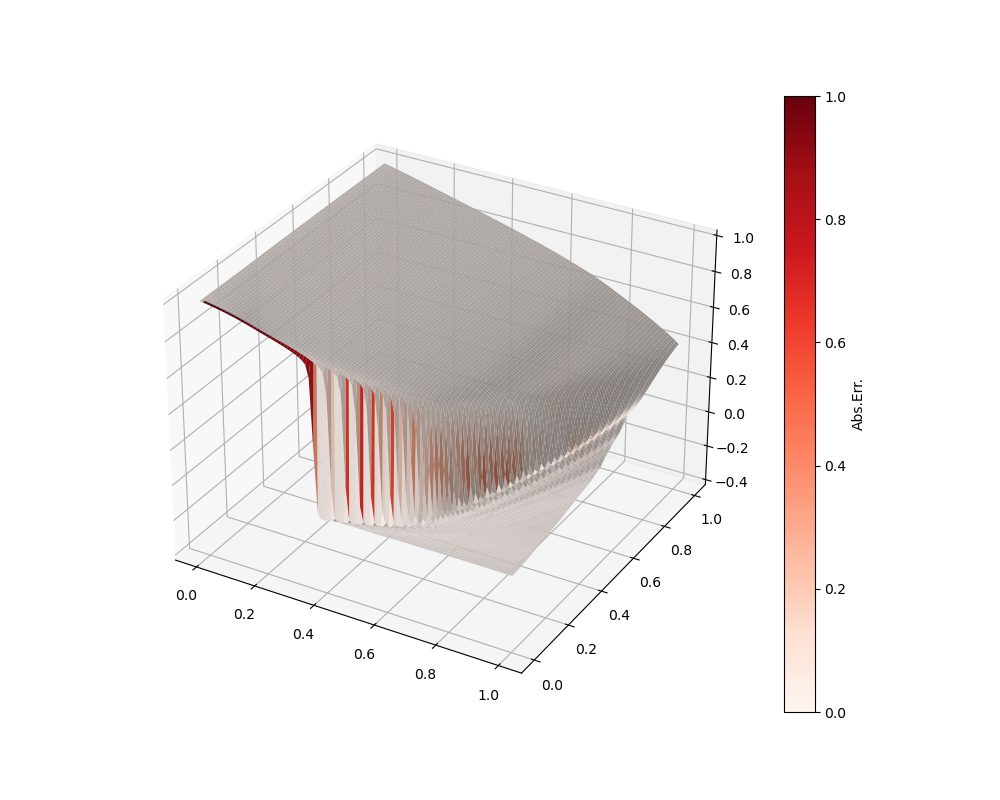}
     }
     \subcaptionbox{VSKs-$f$}{
     \includegraphics[trim={2.3cm 1.6cm 2.45cm 1.9cm},clip,width=0.325\textwidth]{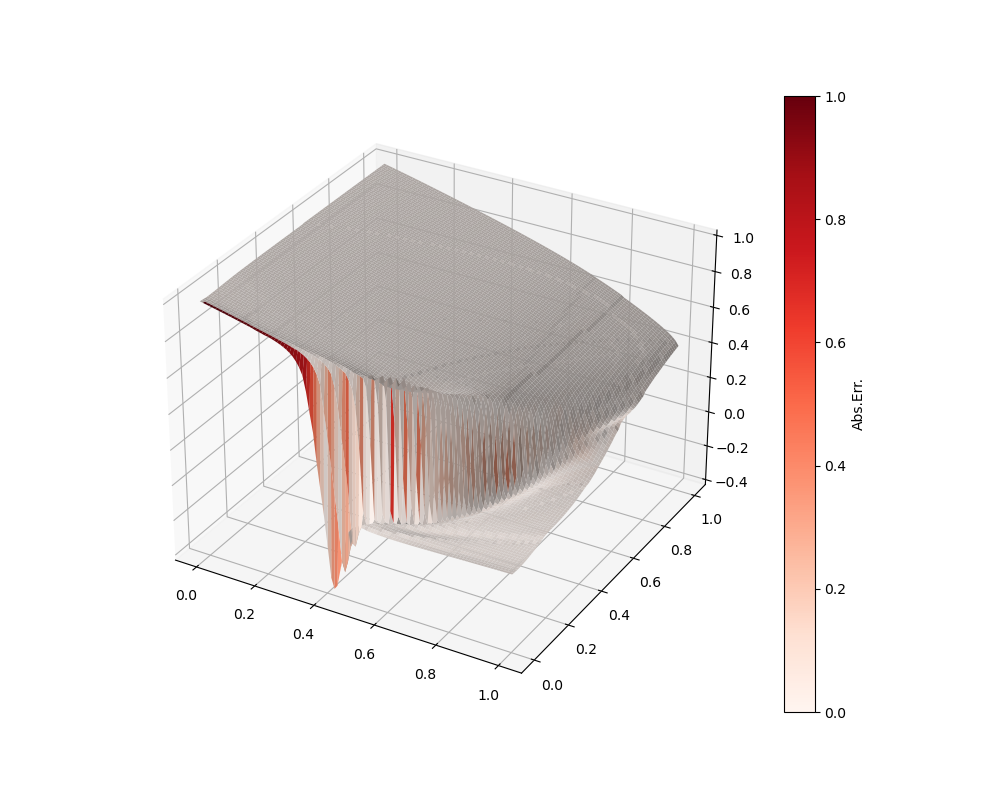}
     }\\
     \subcaptionbox{FSKs}{
     \includegraphics[trim={2.3cm 1.6cm 5.55cm 1.9cm},clip,width=0.28\textwidth]{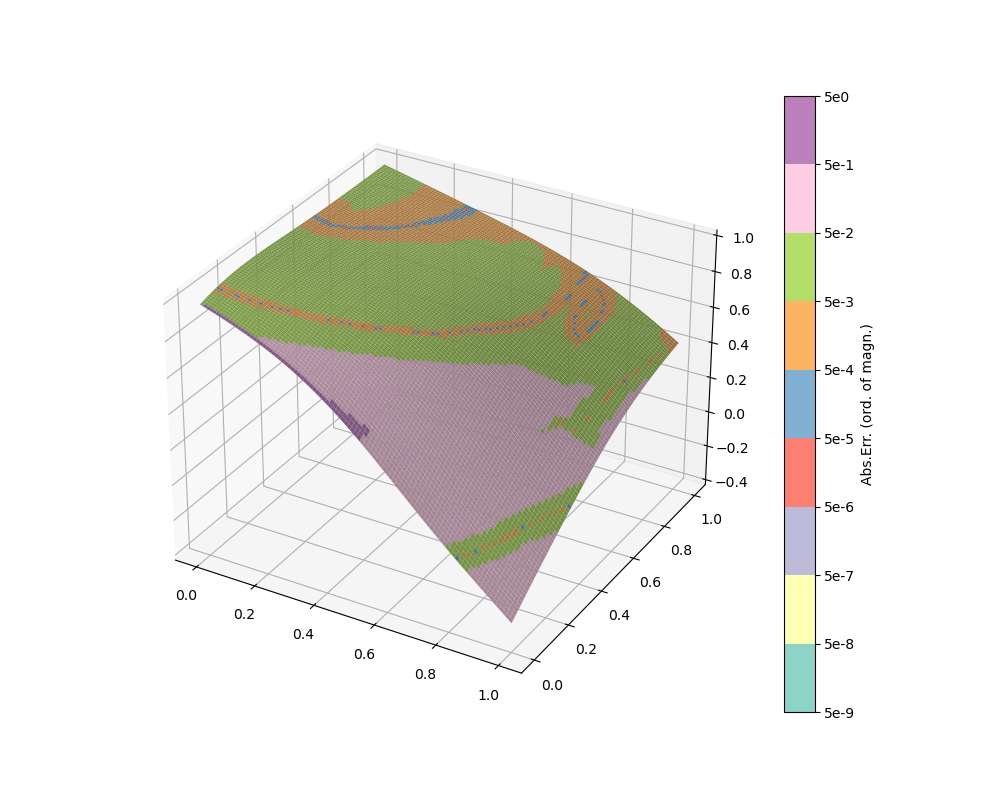}
     }
     \subcaptionbox{$\delta$NN-VSKs}{
     \includegraphics[trim={2.3cm 1.6cm 5.55cm 1.9cm},clip,width=0.28\textwidth]{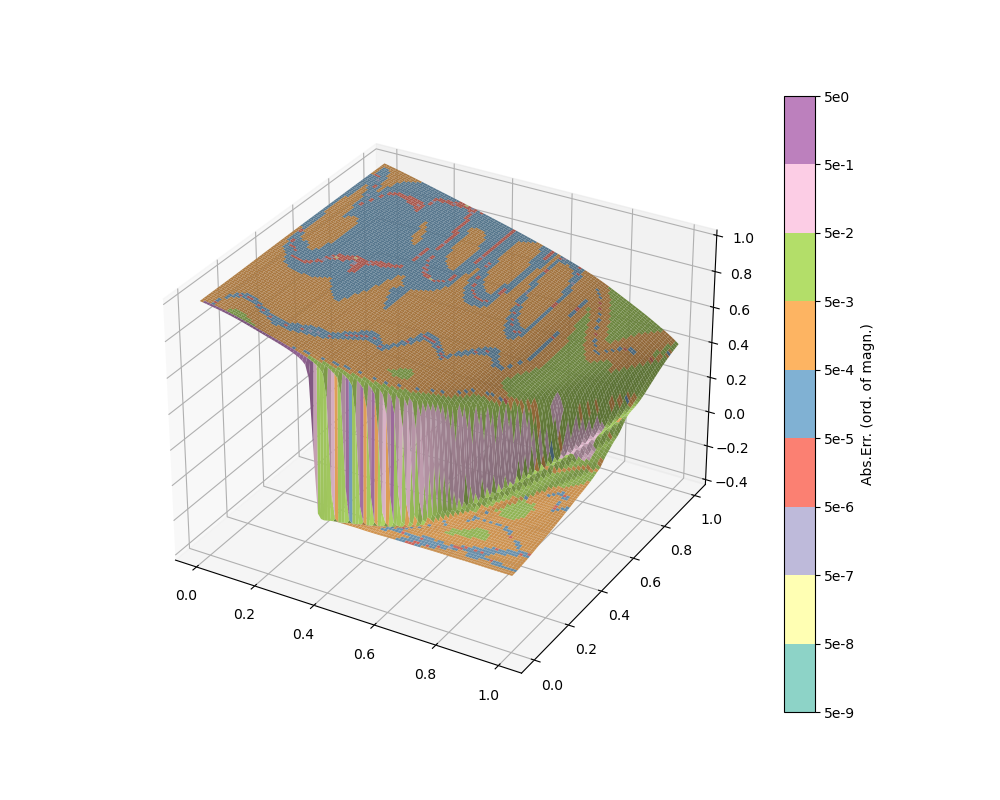}
     }
     \subcaptionbox{VSKs-$f$}{
     \includegraphics[trim={2.3cm 1.6cm 2.45cm 1.9cm},clip,width=0.325\textwidth]{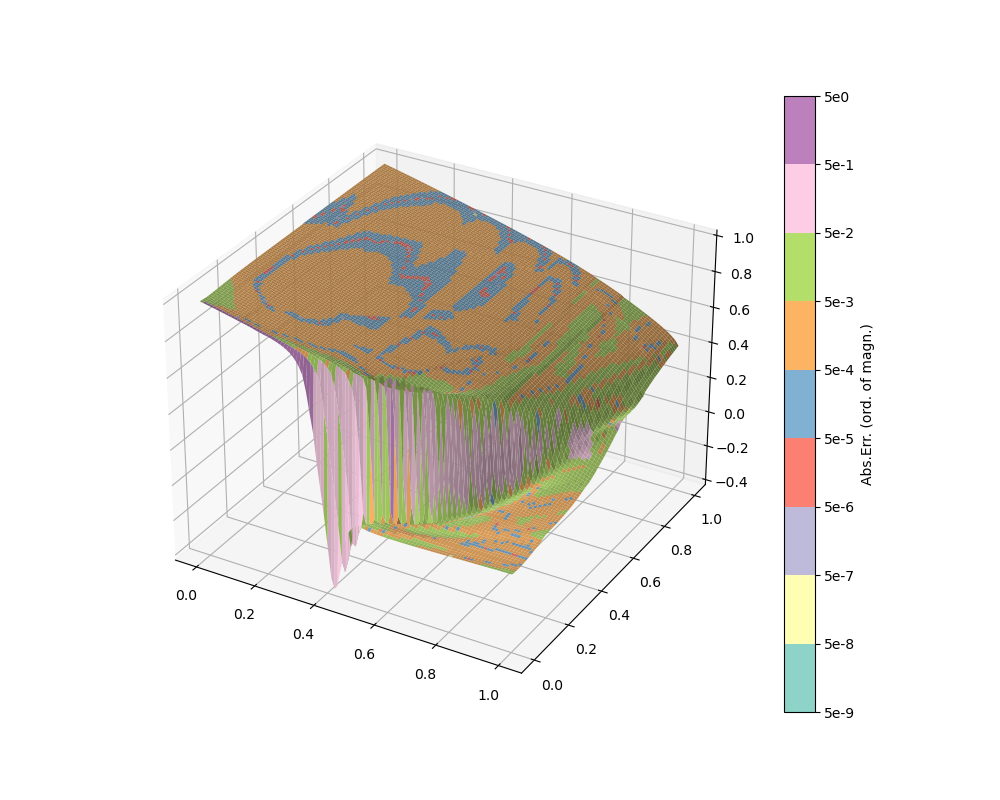}
     }
     \caption{{Interpolation results for $f_4$ (see Figure \ref{fig:f4_topview_comparison}), colored w.r.t. interpolation error (absolute error). In the bottom row, colors describe the order of magnitude of the absolute error. Shared colorscales for subfigures ($b$)-($d$) and for subfigures ($e$)-($g$).}}
     \label{fig:f4_surface_errors_horizontal}
 \end{figure}

\section{Conclusions}
\label{Conclusion}

In this work, by taking advantage of the cardinal form of the VSK interpolant, we introduce new error bounds that relate the accuracy of the VSK interpolant to the definition of the associated scaling function. These bounds point out that the closer the scaling function is to the target, the smaller the pointwise error. We further propose a way to learn the scaling function using $\delta$NNs, and such a data-driven approach offers a user-independent solution for enhancing the flexibility of VSKs in meshfree approximation tasks and numerically shows that the learned scaling function captures the key features of the target, as expected. 

The numerical tests show that our two possible implementations of the VSKs, namely the $\delta$NN-VSKs and the VSKs-$f$, behave better than the classical kernel interpolant. 
The main advantage of the VSKs-$f$ lies in its easy implementation because it directly learns the target which is then used as a scaling function. Nevertheless, the justification for using such a scheme is subordinated to the theoretical results and to the fact that $\delta$NN-VSKs, which are able to learn the nearly-optimal scaling function without any constraints, return scaling functions that truly resemble the targets.

In future work, we plan to learn simultaneously both the scaling function for VSKs and the \emph{optimal} sampling locations producing in that way sparse models (see e.g. \cite{Santin23,Camattari,marchi2005near}). Moreover, the same data-driven learning algorithm could be extended to the context of kernel collocation models for approximating PDEs, with a particular focus on the most local promising techniques as RBF-FD (refer to \cite{Fornberg_Flyer_2015} for a general overview).

\appendix

{
\section{$\delta$NN Architecture in the Numerical Experiments}\label{sec:deltaNN_arch}

In this section, we list the details about the $\delta$NN architecture and the training options used.

\quad

\noindent \textbf{Architecture:}
\begin{enumerate}
    \item Input layer ($\R^2$ inputs);
    \item Two FC layers, 128 units, elu activation function;
    \item\label{item:res} One residual block made of:
    \begin{enumerate}
        \item FC layer, 128 units, elu activation function;
        \item FC layer, 128 units, linear activation function;
        \item Sum of the two layers above;
        \item elu activation function;
    \end{enumerate}
    \item Three blocks of layers, each block made of:
    \begin{enumerate}
        \item FC layer, 128 units, elu activation function;
        \item Residual block as described at item \ref{item:res};
        \item Discontinuous layer, 16 units, elu activation function;
    \end{enumerate}
    \item FC layer, 128 units, elu activation function;
    \item FC layer, 1 unit, linear activation function.
\end{enumerate}

\noindent\textbf{Training Options:}
\begin{itemize}
    \item $\delta$NN-VSKs:
    \begin{itemize}
        \item maximum epochs: 2000;
        \item Adam optimizer \cite{Kingma2015_ADAM}, starting learning rate $10^{-4}$, reduce learning rate on plateau (patience 75 epochs, factor 0.5);
    \end{itemize}
    \item VSKs-$f$:
    \begin{itemize}
        \item maximum epochs: 1000;
        \item Validation set: $20\%$ of the $n$ points;
        \item Mini-batch size: 32;
        \item Adam optimizer \cite{Kingma2015_ADAM}, starting learning rate $10^{-4}$, reduce learning rate on plateau (patience 75 epochs, factor 0.5);
        \item Early stopping (patience 550 epochs).
    \end{itemize}
\end{itemize}


}

\section*{Acknowledgements}
    Gianluca Audone and Emma Perracchione kindly acknowledge the support of the Fondazione Compagnia di San Paolo within the framework of the Artificial Intelligence Call for Proposals, AIxtreme project (ID Rol: 71708). Emma Perracchione further acknowledges the support of the GOSSIP project (``Greedy Optimal Sampling for Solar Inverse Problems'') – funded by the Ministero dell’Universit\`a e della Ricerca (MUR) -  within the PRIN 2022 program, CUP: E53C24002330001.
    Sandra Pieraccini acknowledges the support of the FaReX project (``Full and Reduced order modelling of coupled systems: focus on non-matching methods and automatic learning'') – funded by the Ministero dell’Universit\`a e della Ricerca – within the PRIN 2022 program (D.D.104 - 02/02/2022). Francesco Della Santa and Sandra Pieraccini acknowledge that this study was carried out within the FAIR-Future Artificial Intelligence Research and received funding from the European Union Next-GenerationEU (PIANO NAZIONALE DI RIPRESA E RESILIENZA (PNRR)–MISSIONE 4 COMPONENTE 2, INVESTIMENTO 1.3---D.D. 1555 11/10/2022, PE00000013). This manuscript reflects only the authors’ views and opinions; neither the European Union nor the European Commission can be considered responsible for them.

\end{document}